\documentclass{amsart}

\usepackage{amssymb,amsmath,graphicx}
\usepackage{amsfonts}
\usepackage{amsmath}
\usepackage{amsthm}
\usepackage{amssymb}
\usepackage{graphicx}
\usepackage{comment}
\usepackage[english]{babel}

\DeclareGraphicsExtensions{.pdf,.eps,.ps,.png,.jpg,.jpeg}
\providecommand{\U}[1]{\protect\rule{.1in}{.1in}}
\newtheorem{theorem}{Theorem}[section]

\newtheorem{claim}[theorem]{Claim}

\newtheorem{corollary}[theorem]{Corollary}

\newtheorem{example}[theorem]{Example}

\newtheorem{lemma}[theorem]{Lemma}

\newtheorem{remark}[theorem]{Remark}

\newtheorem*{quote1}{Theorem \ref{orthogonalprojectionforsingleclosedgeodesic}}

\begin{document}
\title{A geometric property of closed geodesics on hyperbolic surfaces}

\author{Max Neumann-Coto}
\address{Instituto de Matem\'aticas, Universidad Nacional Aut\'onoma de M\'exico, Ciudad de M\'exico, 04510 M\'exico}
\email{max.neumann@im.unam.mx}

\author{Peter Scott}
\address{Mathematics Department, University of Michigan
Ann Arbor, Michigan 48109, USA}
\email{pscott@umich.edu}

\thanks{Research partially supported by PAPIIT-UNAM grant IN115414}

\begin{abstract}
\noindent We study closed geodesics on hyperbolic surfaces, and give bounds
for their angles of intersection and self-intersection, and for the sides of
the polygons that they form, depending only on the lengths of
the geodesics.
\end{abstract}

\maketitle

In this paper we consider the geometry of closed geodesics on a hyperbolic
surface $M$, by looking at the geodesic lines that form their pre-images in $%
\mathbb{H}^{2}$. We are interested in knowing how close these lines can be
to each other, how small are their angles of intersection and how large the
polygons formed by these lines can be, in terms of the lengths of the closed
geodesics involved.

If $l$ and $m$ are two geodesic lines in the hyperbolic plane $\mathbb{H}%
^{2} $ that do not share a point at infinity, the orthogonal projection of $%
l $ onto $m$ has finite length, which depends on the distance between the
lines if they are disjoint, or the angle of intersection if they meet. If $l$
and $m$ intersect, the rotation of $\mathbb{H}^{2}$ about their intersection
point through angle $\pi $ interchanges $l$ and $m$. If $l$ and $m$ are
disjoint, we can again interchange them by rotating through $\pi $ about the
midpoint of their common perpendicular. It follows that in all cases, the
orthogonal projection of $l$ onto $m$ and the orthogonal projection of $m$
onto $l$ have the same length.

The main result of the paper is the following.

\begin{quote1}
Let $\gamma $ be a closed geodesic on an orientable hyperbolic surface $M$,
and let $l$ and $m$ be distinct geodesics in $\mathbb{H}^{2}$ above $\gamma$%
. Then the orthogonal projection of $l$ onto $m$ has length strictly less
than $l(\gamma )$.
\end{quote1}

Theorem \ref{orthogonalprojectionforsingleclosedgeodesic} gives a bound for
the self-intersection angles of a hyperbolic geodesic $\gamma $ in terms of
its length $l(\gamma )$. Recall that the \textsl{angle of parallelism} of a
positive number $a$, denoted by $\Pi (a)$, is the third angle of a right
hyperbolic triangle that has two asymptotic parallel sides and a side of
length $a$. In Corollary \ref{bounds for angles}, we show that if $\gamma $
is a closed oriented geodesic on an orientable hyperbolic surface $M$, and $%
\phi$ is the angle formed by the two outgoing arcs of $\gamma$ at a
self-intersection point, then $\Pi (\frac{l(\gamma )}{2})<\phi <\pi -\Pi (%
\frac{l(\gamma )}{4})$.

Theorem \ref{orthogonalprojectionforsingleclosedgeodesic} also gives a bound
for the size of the convex polygons in $\mathbb{H}^{2}$ formed by geodesic
lines above a closed hyperbolic geodesic. In Corollary \ref%
{boundsforpolygons}, we show that if $\gamma $ is a closed geodesic on an
orientable hyperbolic surface $M$, then the triangles formed by the geodesic
lines above $\gamma $ in $\mathbb{H}^{2}$ have sides shorter than $l(\gamma
) $, and the $n$--gons have sides shorter than $(n-2)l(\gamma )$.

The bound given in Theorem \ref{orthogonalprojectionforsingleclosedgeodesic}
is optimal, in the sense that for each hyperbolic surface $M$, there is a
sequence of geodesics in $M$ and lines above them in $\mathbb{H}^{2}$ for
which the ratio of the projection length to the geodesic length approaches $%
1 $. 

This does not imply that the above bounds for the angles are optimal.
Also the above bounds for the sides of $n$--gons might be far from optimal
for $n>3$. 

Theorem \ref{orthogonalprojectionforsingleclosedgeodesic} can be generalized
to the case of two closed geodesics. In Theorem \ref%
{orthogonalprojectionfortwoclosedgeodesics}, we show that if $\gamma $ and $%
\delta $ are closed geodesics on an orientable hyperbolic surface $M$, and
if $l$ and $m$ are distinct geodesics in $\mathbb{H}^{2}$ above $\gamma $
and $\delta $ respectively, then the orthogonal projection of $l$ onto $m$
has length strictly less than $l(\gamma )+l(\delta )$. This bound is optimal
too, in the sense that there are two sequences of closed geodesics $\gamma
_{n}$ and $\delta _{n}$ and lines $l_{n}$ and $m_{n}$ above them, for which
the ratio between the length of the projection of $\gamma _{n}$ to $\delta
_{n}$ and the sum of the lengths of $\gamma _{n}$ and $\delta _{n}$
approaches $1$.

In section \ref{section:Mainresults}, we also give lower bounds for the
intersection angles of $\gamma $ and $\delta $, and we give upper bounds for
the lengths of polygons formed by $\gamma $ and $\delta $.

When we were finishing this paper we learned of a result by Gilman, that is
closely related to Theorem \ref{orthogonalprojectionfortwoclosedgeodesics}.
In her last corollary in \cite{G}, she considers a pair of hyperbolic
isometries that generate a purely hyperbolic subgroup of $SL(2,R)$ and gives
inequalities that imply Theorem \ref%
{orthogonalprojectionfortwoclosedgeodesics}. However, this theorem does not
imply her inequalities, and she obtains somewhat stronger lower bounds for
the angles of intersection of the axes than we obtain. Gilman's result can
also be applied to a pair of conjugate isometries, but Theorem \ref%
{orthogonalprojectionforsingleclosedgeodesic} does not follow from her
result. The bounds we obtain for the angles of self-intersection of a closed
geodesic are substantially stronger than her bounds in this case. Our reason
for keeping Theorem \ref{orthogonalprojectionfortwoclosedgeodesics}, besides
completeness, is that the proof of the most interesting case (when the two
lines meet) is geometric and only uses two easy lemmas.\ At the end of this
paper, we discuss the relations between the bounds obtained from our results
and from Gilman's result in \cite{G}.

Most of this paper is devoted to the proof of Theorem \ref%
{orthogonalprojectionforsingleclosedgeodesic}. We will use a mix of
geometric and algebraic arguments to deal with the different cases, some of
which are very simple and allow us to give some better bounds, but others
are quite intricate and involve many steps. In section \ref{section:trees},
we prove some simple results about axes of elements of a group which acts on
a tree. In section \ref{section:hyperbolic}, we start on the proof of our
main result. We prove some special cases and reduce to the case when $M$ is
a three-punctured sphere or a once-punctured torus. In sections \ref%
{section:3-puncturedsphere} and \ref{section:oncepuncturedtorus}, we
consider these two cases separately. Finally in section \ref%
{section:Mainresults}, we complete the proof of Theorem \ref%
{orthogonalprojectionforsingleclosedgeodesic}, and then deduce several
consequences.

\section{Groups acting on trees\label{section:trees}}

Theorem \ref{orthogonalprojectionforsingleclosedgeodesic} can be regarded as
a geometric version of a simple result about the intersections of axes of
elements of a group acting on a tree.

Let $T$ denote the tree which is the universal cover of a finite graph $X$
with a single vertex. Such a graph is called a \textit{rose}. Thus the free
group $G=\pi _{1}(X)$ acts freely on $T$ with quotient $X$. After orienting
each loop in the rose $X$, there is a naturally associated set $S$ of free
generators of $G$, one generator for each oriented loop. Thus each oriented
edge of $T$ has an element of $S\cup S^{-1}$ associated to it, and any
oriented edge path $\lambda $ in $T$ has a word $w$ in $S\cup S^{-1}$
associated to it. Any element $\alpha $ of $G$ can be expressed uniquely as
a reduced word in $S\cup S^{-1}$. We are interested in the minimal length $%
L(\alpha )$ of all words representing conjugates of $\alpha $. Such a word
realizes this minimal length if and only if it is cyclically reduced. Any
nontrivial element $\alpha $ of $G$ has an axis $A$, which is an edge path
in $T$ that is preserved by $\alpha $, and on which $\alpha $ acts by a
translation of length $L(\alpha )$. This is why we use the same letter, $L$,
for this algebraically defined length as for the length of an edge path in $%
T $. The infinite reduced word associated to $A$ when suitably oriented, is
made by concatenating copies of a cyclically reduced word $w$ equal to a
conjugate of $\alpha$. We call this infinite word the unwrapping of $w$.

\begin{lemma}
\label{vandvinversecannotoverlap} Let $W$ be a reduced word in a free group $%
G$. If $W$ contains two nontrivial subwords $u$ and $v$ and $u=v^{-1}$, then 
$u$ and $v$ cannot overlap.
\end{lemma}

\begin{proof}
We have $u =s_{1} s_{2} \ldots s_{n}$, for some $s_{i}$'s in the generating
set $S \cup S^{ -1}$ of $G$. Thus $v =s_{n}^{ -1} s_{n -1}^{ -1} \ldots
s_{1}^{ -1}$. Without loss of generality we can assume that the overlap
consists of an initial segment $s_{1} \ldots s_{k}$ of $u$ and an ending
segment $s_{k}^{ -1} \ldots s_{1}^{ -1}$ of $v$. Denoting $s_{1} \ldots
s_{k} $ by $z$, we see that $z$ is a nontrivial reduced word which is equal
to its own inverse, which is impossible. This completes the proof of the
lemma.
\end{proof}

In the next lemma we give bounds for the length of the intersection of two
axes in $T$ with conjugate stabilizers.

\begin{lemma}
\label{boundedoverlapbetweenAandgA}Let $G$ be a free group which acts freely
on a tree $T$ with quotient a rose, and naturally associated set $S$ of free
generators of $G$. Let $\alpha $ be a nontrivial element of $G$ with axis $A$%
, and $L (\alpha ) \geq 2$. Let $\alpha ^{ \prime }$ and $\alpha ^{ \prime
\prime }$ be conjugates of $\alpha $ with axes $A^{ \prime }$ and $A^{
\prime \prime }$ respectively. Then the following inequalities hold:

\begin{enumerate}
\item If $A$ and $A^{\prime \prime }$ share an edge of $T$ where the
translation directions of $\alpha $ and $\alpha ^{\prime \prime }$ disagree,
then $L(A\cap A^{\prime \prime })<\frac{L(\alpha )-1}{2}$.

\item If $A$ and $A^{\prime }$ are distinct, then $L(A\cap A^{\prime
})<L(\alpha )-1$.
\end{enumerate}
\end{lemma}

\begin{proof}
1) Choose an orientation of $A$ so that the associated infinite reduced word
is the unwrapping $\widetilde{W}$ of a cyclically reduced word $W$ equal to
a conjugate of $\alpha $. Let $I$ denote the interval $A\cap A^{\prime
\prime }$, regarded as a subinterval of $A$ with the induced orientation,
and let $u$ denote the reduced word associated to $I$. Thus $\widetilde{W}$
contains $u$ as a subword, and also contains infinitely many translates of $%
u $ by shifting by powers of $W$. Next consider the axis $A^{\prime \prime }$
of $\alpha ^{\prime \prime }$. As $\alpha ^{\prime \prime }$ is conjugate to 
$\alpha $, we know $A^{\prime \prime }$ is a translate of $A$, and we give $%
A^{\prime \prime }$ the induced orientation. Now the infinite reduced word
associated to $A^{\prime \prime }$ is also equal to $\widetilde{W}$, and so
also contains infinitely many copies of the subword $u$. If we cycle $W$ to
begin with $u$, we have $W=uv$, for some reduced word $v$, and $\widetilde{W}%
=\ldots uvuvuv\ldots $.

As the translation directions of $\alpha $ along $A$ and of $\alpha ^{\prime
\prime }$ along $A^{\prime \prime }$ disagree, the infinite reduced word $%
\widetilde{W}$ must contain a copy of $u^{-1}$. Lemma \ref%
{vandvinversecannotoverlap} tells us that $u$ and $u^{-1}$ cannot overlap,
so that $u^{-1}$ must be a subword of the subword $v$ of $\widetilde{W}$.
Further $u^{-1}$ cannot contain the initial or the final letter of $v$, as
this would contradict the fact that $\widetilde{W}$ is reduced. As $u$ and $%
u^{-1}$ have the same length, we conclude that $2L(u)\leq L(W)-2$, so that $%
L(u)\leq \frac{L(w)-2}{2}=\frac{L(\alpha )-2}{2}$.

2) Let $I$ denote the interval $A\cap A^{\prime }$. We will assume that $%
L(I)\geq 1$, as otherwise the result is trivial. Further we can assume that
the translation directions of $\alpha $ and $\alpha ^{\prime }$ agree on $I$%
, as otherwise the result follows from part 1).

If $I$ has length at least $L(\alpha )$, and $x$ denotes the initial point
of $I$, then $\alpha x$ and $\alpha ^{\prime }x$ must be equal as each is a
vertex of $I$ with distance $L(\alpha )$ from $x$. Note that this uses our
assumption that the translation directions of $\alpha $ and $\alpha ^{\prime
}$ agree. It follows that $\alpha $ equals $\alpha ^{\prime }$, so that $A$
equals $A^{\prime }$, which contradicts our hypothesis. Thus we will suppose
that $I$ has length $L(\alpha )-1$, and let $x$ and $y$ denote the endpoints
of $I$. Now the translation lengths of $\alpha $ along $A$ and of $\alpha
^{\prime }$ along $A^{\prime }$ are equal to $L(\alpha )$. Thus, by
interchanging $x$ and $y$ if needed, we have that $\alpha x$ and $\alpha
^{\prime }x$ each has distance $1$ from $y$. Thus $\alpha ^{\prime -1}\alpha 
$ moves $x$ distance $2$. But $\alpha ^{\prime -1}\alpha $ is a commutator,
as $\alpha ^{\prime }$ is a conjugate of $\alpha $, and so either it is
trivial or moves any point distance at least $4$. This is because a reduced
word of length strictly less than $4$ must be trivial or map to a non-zero
element in the abelianisation of $G$. We conclude that $\alpha ^{\prime
-1}\alpha $ is trivial, so that $\alpha $ and $\alpha ^{\prime }$ are equal
and $A$ is equal to $A^{\prime }$, which again contradicts our hypothesis.
This contradiction shows that $L(I)<L(\alpha )-1$, which completes the proof
of part 2) of the lemma.
\end{proof}

The following examples show that the inequalities given in Lemma \ref%
{boundedoverlapbetweenAandgA} are sharp. Let $G$ be the free group of rank $%
2 $ generated by $x$ and $y$, and let $T$ be the standard $4$-valent tree
with vertices labelled by elements of $G$.

For part 1), we let $\alpha =xy^{-1}xy$ and $\alpha ^{\prime \prime
}=x^{-1}\alpha x=y^{-1}xyx$, so that $L(\alpha )=4$, and $\frac{L(\alpha )-1%
}{2}=\frac{3}{2}$. Thus $L(A\cap A^{\prime \prime })<\frac{3}{2}$, and we
claim that $L(A\cap A^{\prime \prime })=1$. As $\alpha $ and $\alpha
^{\prime \prime }$ are cyclically reduced, each of their axes, $A$ and $%
A^{\prime \prime }$, passes through the vertex $e$ of $T$. Now it is easy to
check that $A$ and $A^{\prime \prime }$ each contain the edge of $T$ with
vertices $e$ and $y^{-1}$, and that indeed the translation directions of $%
\alpha $ along $A$ and of $\alpha ^{\prime \prime }$ along $A^{\prime \prime
}$ disagree on this edge. Hence $L(A\cap A^{\prime \prime })=1$, as required.

For part 2) we let $\alpha =xyx$ and $\alpha ^{\prime }=x\alpha x^{-1}=xxy$,
so that $L(\alpha )=3$, and $L(\alpha )-1=2$. Thus $L(A\cap A^{\prime })<2$,
and we claim that $L(A\cap A^{\prime })=1$. As $\alpha $ and $\alpha
^{\prime }$ are cyclically reduced, each of their axes, $A$ and $A^{\prime }$%
, passes through the vertex $e$ of $T$. Now it is easy to check that $A$ and 
$A^{\prime }$ each contain the edge of $T$ with vertices $e$ and $x$. Hence $%
L(A\cap A^{\prime })=1$, as required.

\section{Hyperbolic geometry\label{section:hyperbolic}}

Let $M$ be an orientable hyperbolic surface and let $\gamma $ be a closed
geodesic on $M$. The universal cover $\tilde{M}$ of $M$ is isometric to the
hyperbolic plane $\mathbb{H}^{2}$. If $l$ and $m$ are geodesic lines in $%
\mathbb{H}^{2}$ above $\gamma $, there is $g$ in $\pi _{1}(M)$ such that $%
m=gl$. Let $\alpha $ denote a generator of the stabilizer of $l$, and
consider the cover $F$ of $M$ with fundamental group generated by $\alpha $
and $g$. As $F$ is hyperbolic, and $\pi _{1}(F)$ has two generators, $F$
cannot be closed. As $F$ is also orientable, it follows that it must be
homeomorphic to a sphere with three disjoint discs removed, or a torus with
a disc removed. Note that as $g$ lies in $\pi _{1}(F)$, the geodesics $l$
and $m$ in $\mathbb{H}^{2}$ project to a single closed geodesic on $F$ with
the same length as $\gamma $. Hence, in order to prove Theorem \ref%
{orthogonalprojectionforsingleclosedgeodesic}, it suffices to handle the
case when $M$ is equal to $F$.

Geometrically there are several distinct possibilities depending on whether
each end of $F$ is a cusp or contains a closed geodesic. For simplicity in
most of our arguments below, we will consider only the special cases when
the ends of $F$ are cusps. It turns out that these are the most subtle
cases. In section \ref{section:Mainresults}, we will discuss how to prove
Theorem \ref{orthogonalprojectionforsingleclosedgeodesic} in general using
essentially the same arguments as in these special cases.

If $F$ is a three-punctured sphere or a once-punctured torus, there is a
pair of disjoint simple infinite geodesics which together cut $F$ into an
ideal quadrilateral $Q$. The three-punctured sphere is the double of an
ideal triangle, and we choose two of the common edges of the triangles to be
the geodesics which cut $F$ into the ideal quadrilateral $Q$. The third
common edge becomes one of the diagonals of $Q$. As the three-punctured
sphere admits a reflection isometry which fixes the three common edges, it
follows that $Q$ admits a reflection isometry in this diagonal. Hence the
diagonals of $Q$ must meet at right angles. If $F$ is a once-punctured
torus, $Q$ may not admit any reflection isometries.

\begin{figure}[ptb]
\centering 
\includegraphics[height=1.1018in]{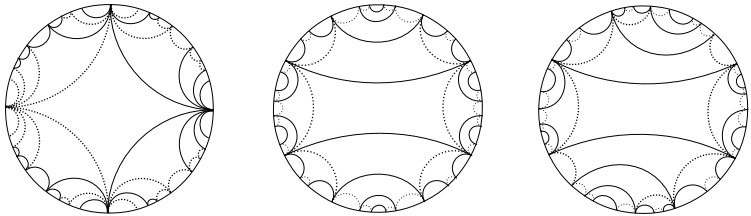}
\caption{{\protect\small Tessellations of $H^{2}$ by quadrilaterals from a
three-punctured sphere and from different once-punctured tori.}}
\label{tessellations}
\end{figure}

Now we consider the universal cover $\tilde{F}$ of $F$, which is isometric
to $\mathbb{H}^{2}$, and is naturally tiled by copies of the quadrilateral $%
Q $ as in Figure \ref{tessellations}. The tessellation obtained from a
three-punctured sphere is very symmetric, but the tessellations obtained
from a once-punctured torus need not be so symmetric. We will consider the
tree $T$ dual to these edges. The geodesics $l$ and $m$ are automatically
transverse to these cutting edges, and will intersect each translate of $Q$
in some (possibly empty) collection of embedded arcs. The group $\pi _{1}(F)$
acts freely on $T$ with quotient a finite graph with a single vertex. There
is a natural projection of $\tilde{F}$ onto $T$ which maps a (thin) collar
neighborhood of each cutting geodesic onto an edge of $T$, and maps the rest
of each polygon to a vertex of $T$. The projections of the geodesics $l$ and 
$m$ to $T$ are injective on the intersection with each collar neighborhood
of a cutting arc. As two distinct geodesics in $\mathbb{H}^{2}$ cross at
most once, it follows that the images of $l$ and $m$ traverse each edge of $%
T $ at most once. Thus the images in $T$ of $l$ and $m$ are the axes of $%
\alpha $ and $g\alpha g^{-1}$ respectively.

Now we are in a position to apply our results from section \ref%
{section:trees}, but we will first consider some easy cases of Theorem \ref%
{orthogonalprojectionforsingleclosedgeodesic}, which can be settled using
only geometric arguments.

We will say that the \textit{bisector} of two disjoint geodesics $\lambda $
and $\lambda ^{\prime }$ in $\mathbb{H}^{2}$ is the unique geodesic $\Lambda 
$ such that reflection in $\Lambda $ interchanges $\lambda $ and $\lambda
^{\prime }$.

\begin{lemma}
\label{disjointgeodesicsbisector} Let $\gamma $ be a closed geodesic on an
orientable hyperbolic surface $M$, and let $l$ and $m$ be disjoint geodesics
in $\mathbb{H}^{2}$ above $\gamma $. Then the orthogonal projection of the
bisector of $l$ and $m$ onto $m$ has length no greater than $l(\gamma )$.
\end{lemma}

\begin{proof}
Let $b$ be the bisector of $l$ and $m$. We will show that the translates of $%
b$ by the action of the stabilizer of $l$ are disjoint. See Figure \ref%
{disjoint_geodesics}.

If $\alpha $ is a generator of the stabilizer of $l$, and $r$ is the
reflection of $\mathbb{H}^{2}$ in $b$, then $r^{\prime }=\alpha \circ r\circ
\alpha ^{-1}$ is the reflection of $\mathbb{H}^{2}$ in $\alpha b$. If $b$
and $\alpha b$ were not disjoint, then $r^{\prime }\circ r$ would fix their
intersection point, and so $r^{\prime }\circ r$ would be an elliptic
isometry. But this is impossible, because $r^{\prime }\circ r$ is an element
of $\pi _{1}(M)$. To prove this, observe that $r^{\prime }\circ r$ maps $m$
to $\alpha m$ preserving the orientations induced by $\gamma $. If $p$ and $%
p^{\prime }$ are the closest points of the geodesics $m$ and $l$, then $r$
maps $p$ to $p^{\prime }$, and $r^{\prime }$ maps $\alpha p^{\prime }$ to $%
\alpha p$. So $r^{\prime }\circ r$ maps the point $p$ in $m$ to a point in $%
\alpha m$ at distance $l(\gamma )$ from $\alpha p$. Therefore $\alpha
^{-1}\circ r^{\prime }\circ r$ is a translation along $m$ of length $%
l(\gamma )$, and so it lies in $\pi _{1}(M)$.

Let $\mu $ be the geodesic joining $p$ and $p^{\prime }$, so $\mu $ crosses $%
l$, $m$ and $b$ orthogonally. Thus $\alpha \mu $ joins $\alpha p$ and $%
\alpha p^{\prime }$ and crosses $\alpha l$, $\alpha m$ and $\alpha b$
orthogonally. Now if $\lambda $ is the perpendicular bisector of the
geodesic segment joining $p^{\prime }$ and $\alpha p^{\prime }$, then
reflection in $\lambda $ preserves $l$ and interchanges $b$ and $\alpha b$.
It follows that $\lambda $ is disjoint from $b$ and $\alpha b$, because if $%
\lambda $ met $b$, it would have to meet $\alpha b$ in the same point,
contradicting the fact that $b$ and $\alpha b$ are disjoint. Hence $\lambda $
crosses $l$ orthogonally and separates $b$ from $\alpha b$. It follows that $%
b$ lies between $\lambda $ and $\alpha \lambda $, so the orthogonal
projection of $b$ onto $l$ has length no greater than $l(\gamma )$, as
required.
\end{proof}

\begin{figure}[ptb]
\centering
\includegraphics[ height=1.1in]{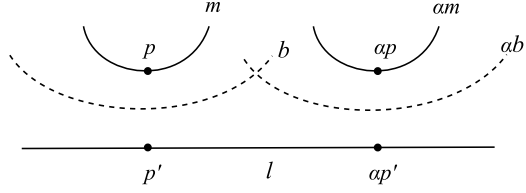}
\caption{{\protect\small If $l$ and $m$ are disjoint geodesics.}}
\label{disjoint_geodesics}
\end{figure}

\begin{example}
The bound given in the previous lemma is sharp. If $M$ is a once-punctured
torus for which the tessellation of $\mathbb{H}^{2}$ by quadrilaterals is
symmetric, and $\gamma $ is a longitude of $M$, then the bisector of two
adjacent geodesics $l$ and $m$ above $\gamma $ is an edge of a
quadrilateral, which projects to an arc of $l$ of length $l(\gamma )$.
\end{example}

At this point we need a brief discussion of orientations of geodesics in $%
\mathbb{H}^{2}$. If two oriented geodesics in $\mathbb{H}^{2}$ are disjoint,
it makes sense to say that they are coherently or oppositely oriented. Two
oriented disjoint geodesics $l$ and $m$ are \textit{coherently oriented} if,
for any geodesic $\lambda $ which cuts both of them, they cross $\lambda $
in the same direction, and they are \textit{oppositely oriented} if, for any
geodesic $\lambda $ which cuts both of them, they cross $\lambda $ in
opposite directions. Clearly this does not depend on the geodesic $\lambda $%
. In particular, $l$ and $m$ are coherently oriented if and only if the
orthogonal projection of $l$ onto $m$ is coherently oriented with $m$. (Note
that the orthogonal projection of $l$ onto $m$ inherits a natural
orientation from that of $l$, as it cannot consist of a single point, unless 
$l$ and $m$ cross and do so orthogonally.) If we choose an orientation of a
closed geodesic $\gamma $ on a hyperbolic surface $M$, it induces an
orientation of each geodesic in $\mathbb{H}^{2}$ above $\gamma $. If $l$ and 
$m$ are two disjoint such geodesics which are coherently oriented, they will
remain coherently oriented if we reverse the orientation of $\gamma $, and
the same applies if they are oppositely oriented. Thus we do not need to
specify an orientation for $\gamma $ in order to ask the question whether $l$
and $m$ are coherently or oppositely oriented.

If $l$ and $m$ are crossing oriented geodesics in $\mathbb{H}^{2}$, the
above definition of coherent orientation does not work, as the way they
cross $\lambda $ does depend on $\lambda $. But still the orthogonal
projection of $l$ onto $m$ is coherently oriented with $m$ if and only if
the orthogonal projection of $m$ onto $l$ is coherently oriented with $l$.
(Unless $l$ and $m$ cross at right angles.)

Having discussed relative orientations for geodesics in $\mathbb{H}^{2}$, we
also need to make clear the connection between this and the relative
orientations of the corresponding axes in the tree $T$ described above which
is dual to some family of cutting geodesics in $\mathbb{H}^{2}$. Let $%
\lambda $ denote one of these cutting geodesics which meets disjoint
geodesics $l$ and $m$ above a closed geodesic $\gamma $ in a hyperbolic
surface $M$. If $e$ denotes the edge of $T$ dual to $\lambda $, the axes in $%
T$ to which $l$ and $m$ project must each contain $e$, and the directions in
which $l$ and $m$ cross $\lambda $ are the same as the direction in which
the axes traverse $e$. Thus if $l$ and $m$ are disjoint geodesics in $%
\mathbb{H}^{2}$ above $\gamma $, and if their images $A$ and $B$ in $T$ have
a common edge, then $l$ and $m$ are coherently oriented if and only if $A$
and $B$ are coherently oriented on their intersection. However if $l$ and $m$
are crossing geodesics above $\gamma $, and if the axes $A$ and $B$ in $T$
overlap, these axes may or may not be coherently oriented, and this may well
depend on the choice of cutting geodesics in $\mathbb{H}^{2}$. In
particular, it is possible that the orthogonal projection of $l$ onto $m$ is
coherently oriented with $m$, while $A$ and $B$ overlap and have opposite
orientations.

Having completed this discussion, we can give a special case of Theorem \ref%
{orthogonalprojectionforsingleclosedgeodesic} for which we get an even lower
bound on the orthogonal projection of $l$ onto $m$. This is analogous to the
result of part 1) of Lemma \ref{boundedoverlapbetweenAandgA}.

\begin{figure}[ptb]
\centering
\includegraphics[ height=0.9in]{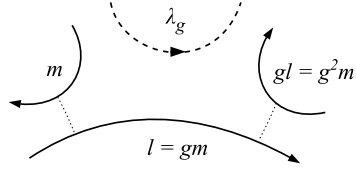}
\caption{{\protect\small $l$ and $m$ are disjoint and oppositely oriented
geodesics.}}
\label{disjoint_oppositely_oriented_geodesics}
\end{figure}

\begin{lemma}
\label{disjointoppositelyorientedprojection} Let $\gamma $ be a closed
geodesic on an orientable hyperbolic surface $M$, and let $l$ and $m$ be
disjoint geodesics in $\mathbb{H}^{2}$ above $\gamma $ which are oppositely
oriented. Then the orthogonal projection of $l$ onto $m$ has length strictly
less than $l(\gamma )/2$.
\end{lemma}

\begin{remark}
\label{nocommonendpoint}The fact that $\pi _{1}(M)$ is a discrete group of
isometries of  $\mathbb{H}^{2}$ implies that two distinct geodesics in $%
\mathbb{H}^{2}$ above $\gamma $ cannot have a common endpoint at infinity.
This fact is needed in the proof of this and many later results.
\end{remark}

\begin{proof}
As $l$ and $m$ both lie above $\gamma $, there is $g$ in $\pi _{1}(M)$ such
that $l=gm$. Let $g$ denote any element of $\pi _{1}(M)$ such that $l=gm$,
and let $\lambda _{g}$ denote the axis of $g$ in $\mathbb{H}^{2}$. Clearly $%
\lambda _{g}$ meets $l$ if and only if $\lambda _{g}$ meets $m$. As $l$ and $%
m$ are oppositely oriented, they must be disjoint from $\lambda _{g}$. Now
consideration of the various possibilities shows that $\lambda _{g}$ must
lie in the region of $\mathbb{H}^{2}$ between $l$ and $m$, and must not
separate $l$ from $m$. See Figure \ref%
{disjoint_oppositely_oriented_geodesics}. In particular, it is now clear
that $m$ cannot meet $gl=g^{2}m$. This holds for any $g$ such that $l=gm$.
Thus $m$ is also disjoint from $\alpha ^{k}gl$, for each integer $k$, where $%
\alpha $ generates the stabilizer of $l$. It follows immediately that the
geodesics in the family $\{\alpha ^{k}m,\alpha ^{k}gl\}_{k\in \mathbb{Z}}$,
are all disjoint, and disjoint from $l$. Now we consider the orthogonal
projections onto $l$ of these geodesics. These must all be disjoint, as any
two of them can be separated by a geodesic perpendicular to $l$, constructed
as in the proof of the previous lemma. The projection of $gl$ to $l$ is the
translate by $g$ of the projection of $l$ to $g^{-1}l=m$, and so has the
same length. We conclude that the orthogonal projections onto $l$ of the
geodesics in the family will all have the same length. Now it follows that
this length can be at most $l(\gamma )/2$. Equality can only occur if there
are distinct geodesics in the family $\{\alpha ^{k}m,\alpha ^{k}gl\}_{k\in 
\mathbb{Z}}$, with a common endpoint at infinity. This is not possible, by
Remark \ref{nocommonendpoint}, so it follows that the orthogonal projection
of $m$ onto $l$ has length strictly less than $l(\gamma )/2$, as required.
\end{proof}

Next we will obtain analogous results in the case of crossing geodesics,
even though we can no longer compare orientations of such geodesics in the
same way.

\begin{lemma}
\label{oppositelyorientedprojection}Let $\gamma $ be a closed geodesic on an
orientable hyperbolic surface $M$, and let $l$ and $m$ be crossing geodesics
in $\mathbb{H}^{2}$ above $\gamma $. If the orthogonal projection of $m$ to $%
l$ has opposite orientation from that of $l$, then this projection has
length strictly less than $l(\gamma)/2 $.
\end{lemma}

\begin{proof}
The orientation of $\gamma $ gives orientations for $l$ and $m$ and their
translates. Let $l_{-}$ and $l_{+}$ be the points at infinity of $l$, and
let $m_{-}$ and $m_{+}$ be the points at infinity of $m$ in the given
directions.

Let $\alpha $ be the generator of the stabilizer of $l$, and let $g$ be a
covering translation that maps $m$ to $l$, so $g$ maps the point $p=m\cap l$
to the point $p^{\prime }=l\cap gl$. We can assume (by composing $g$ with
some power of $\alpha $ if necessary) that $p^{\prime }$ lies between $p$
and $\alpha p$. See Figure \ref{crossing_oppositely_oriented_geodesics}a,
where the positive ends of the lines correspond to the tips of the arrows.

We will show that $m=g^{-1}l$ and $m^{\prime }=gl$ are disjoint. By
hypothesis, the angle $m_{-}pl_{+}$, which is equal to the angle $%
l_{-}p^{\prime }m_{+}^{\prime }$, is less than $\pi /2$. Thus the rays $%
pm_{+}$ and $p^{\prime }m_{-}^{\prime }$ cannot cross.

Now we will show that the rays $pm_{-}$ and $p^{\prime }m_{+}^{\prime }$
also cannot cross. Let $\lambda $ be the bisector of the angle $m_{-}pl_{+}$%
, so $g\lambda $ is the bisector of the angle $l_{-}p^{\prime }m_{+}^{\prime
}$. Then $\lambda $ and $g\lambda $ must be disjoint. For otherwise they
would meet at a point $x$ that is equidistant from $p$ and $p^{\prime }$,
and as $g$ is orientation preserving and maps $\lambda $ to $g\lambda $
sending $p$ to $p^{\prime }$, it would have to fix $x$, contradicting the
fact that $g$ cannot have fixed points. As $\lambda $ and $g\lambda $ are
disjoint, they separate the rays $pm_{-}$ and $p^{\prime }m_{+}^{\prime }$
so these rays cannot meet. Hence $m=g^{-1}l$ and $m^{\prime }=gl$ are
disjoint, as required.

Now the covering translation $\alpha g^{-1}$ maps $p^{\prime }=m^{\prime
}\cap l$ to $\alpha p=l\cap \alpha m$ so it maps $m^{\prime }=gl$ to $\alpha
l$, which is equal to $l$, and the preceding argument shows that $m^{\prime }
$ is disjoint from $\alpha m$. Thus $m$ and its translates by powers of $%
\alpha $ are disjoint from $m^{\prime }$ and its translates by powers of $%
\alpha $. As the projection of $m^{\prime }$ to $l$ has the same length as
the projection of $m$ to $l$ (because $m$ and $m^{\prime }$ make the same
angle with $l$), and the sum of the two lengths is less than the translation
length of $\alpha $, it follows that the projection of $m$ to $l$ is shorter
than $l(\gamma )/2$. Again we need Remark \ref{nocommonendpoint} to ensure
this inequality is strict.
\end{proof}

\begin{figure}[ptb]
\centering
\includegraphics[ height=1.3in]{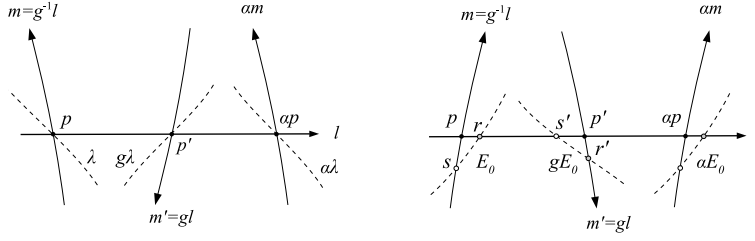}
\caption{{\protect\small a) Projection is oppositely oriented \hfill b)
Overlap is oppositely oriented}}
\label{crossing_oppositely_oriented_geodesics}
\end{figure}

\begin{lemma}
\label{oppositelyorientedaxesinT}Let $\gamma $ be a closed geodesic on an
orientable hyperbolic surface $M$, and let $l$ and $m$ be crossing geodesics
in $\mathbb{H}^{2}$ above $\gamma $ whose images $A$ and $B$ in $T$ overlap
with opposite orientations. Then the orthogonal projection of $m$ onto $l$
has length strictly less than $l(\gamma )/2$.
\end{lemma}

\begin{proof}
We will use the same notation as in the proof of the previous lemma. Let $%
\alpha $ be the generator of the stabilizer of $l$, and let $g$ be a
covering translation that maps $m$ to $l$ and sends the point $p=m\cap l$ to
the point $p^{\prime }=l\cap gl$ that lies between $p$ and $\alpha p$. We
want to show that $m=g^{-1}l$ and $m^{\prime }=gl$ are disjoint. For then
the same argument as in the proof of Lemma \ref{oppositelyorientedprojection}
shows that the projection of $m$ to $l$ is shorter than $l(\gamma )/2$.

We can assume that the angle $m_{-}pl_{+}$, which is equal to the angle $%
l_{-}p^{\prime }m_{+}^{\prime }$, is greater than $\pi /2$, as otherwise the
previous lemma gives the result. So the rays $pm_{-}$ and $p^{\prime
}m_{+}^{\prime }$ cannot cross. See Figure \ref%
{crossing_oppositely_oriented_geodesics}b.

It remains to show that the rays $pm_{+}$ and $p^{\prime }m_{-}^{\prime }$
cannot cross. By hypothesis, the axes $A$ and $B$ in $T$ overlap with
opposite orientations, so there is an edge $e_{0}$ of $T$ that is traversed
by $A$ and $B$ in different directions. If $E_{0}$ is the cutting geodesic
dual to $e_{0}$, then $l$ and $m$ cross $E_{0}$ in opposite directions. Thus 
$gE_{0}$ is also a cutting geodesic, and $E_{0}$ and $gE_{0}$ are disjoint
(they cannot be equal because they cross $l$ in opposite directions).

If $E_{0}$ crosses $l$ and $m$ at $p$ then $gE_{0}$ crosses $l$ and $%
m^{\prime }$ at $p^{\prime }$ and the arcs $E_{0}$ and g$E_{0}$ play the
roles of the bisectors $\lambda $ and $g\lambda $ in Figure \ref%
{crossing_oppositely_oriented_geodesics}a. As $E_{0}$ and g$E_{0}$ are
disjoint, they must separate the rays $pm_{+}$ and $p^{\prime }m_{-}^{\prime
}$, which are therefore disjoint.

If $E_{0}$ crosses $l$ and $m$ at two different points $r$ and $s$, then
either $r$ is in the ray $pl_{-}$and $s$ is in the ray $pm_{+}$, or $r$ is
in the ray $pl_{+}$ and $s$ is in the ray $pm_{-}$. See Figure \ref%
{crossing_oppositely_oriented_geodesics}b. In either case as $E_{0}$ and g$%
E_{0}$ are disjoint, $E_{0}$ cannot cross $m^{\prime }$ and $gE_{0}$ cannot
cross $m$, so $E_{0}$ and g$E_{0}$ must again separate the ends of the rays $%
pm_{+}$ and $p^{\prime }m_{-}^{\prime }$, so that these two rays must be
disjoint.

The above argument shows that, in all cases, $m=g^{-1}l$ and $m^{\prime }=gl$
are disjoint, as required.
\end{proof}

\begin{lemma}
Let $\gamma $ be a closed geodesic on an orientable hyperbolic surface $M$,
and let $l$ and $m$ be crossing geodesics in $\mathbb{H}^{2}$ above $\gamma $%
. Then the orthogonal projections of the bisectors of $l$ and $m$ onto $m$
have lengths whose sum is at most $2l(\gamma )$.
\end{lemma}

\begin{proof}
Let $\alpha $ be a generator of the stabilizer of $l$, and let $\beta $ be
the corresponding generator of the stabilizer of $m$, so that they are
conjugate in $\pi _{1}(M)$. Let $\lambda $ and $\lambda ^{\prime }$ denote
the bisectors of $l$ and $m$, as shown in Figure \ref%
{crossing_geodesics_bisectors}. We will show that the geodesics $\lambda $, $%
\alpha \lambda ^{\prime }$ and $\alpha ^{2}\lambda $ are disjoint, so that
their orthogonal projections to $l$ are also all disjoint. This will imply
that the sum of the orthogonal projections of $\lambda $ and $\alpha \lambda
^{\prime }$ to $l$ is at most the translation length of $\alpha ^{2}$, which
is $2l(\gamma )$, as required.

Let $p$ be the image of $l\cap m$ in $\gamma $. Consider the commutator $%
\kappa =\alpha \beta \alpha ^{-1}\beta ^{-1}$, which must be a hyperbolic or
a parabolic isometry of $\mathbb{H}^{2}$. Thus $\kappa $ is represented by a
loop $\delta $ on $M$ that follows $\gamma $ turning left or right (but
always to the same side) each time that it reaches $p$, so $\delta $ covers
the image of $\gamma $ four times before closing up. The loop $\delta $ is
the projection of a piecewise geodesic line $k$ in $\mathbb{H}^{2}$ formed
by an arc of $\alpha ^{-1}l$, an arc of $m$, an arc of $l$, an arc of $%
\alpha m$ and their translates by the action of $\kappa =\alpha \beta \alpha
^{-1}\beta ^{-1}$. See Figure \ref{crossing_geodesics_bisectors}. The
bisectors at the corners of $k$ are $\beta \lambda ^{\prime }$, $\lambda $, $%
\alpha \lambda ^{\prime }$, $\alpha \beta \lambda $ and their translates by
powers of $\kappa =\alpha \beta \alpha ^{-1}\beta ^{-1}$.

By the symmetry of the picture, if $\lambda $ were to intersect $\alpha
\lambda ^{\prime }$, then $\beta \lambda ^{\prime }$ would intersect $\alpha
\lambda ^{\prime }$ at the same point, and so all the bisectors at the
corners of $k$ would cross at that point, which would be fixed by $\kappa
=\alpha \beta \alpha ^{-1}\beta ^{-1}$ and so $\kappa $ would be elliptic, a
contradiction.
\end{proof}

\ 

\begin{figure}[ptb]
\centering 
\includegraphics[ height=1.4in]{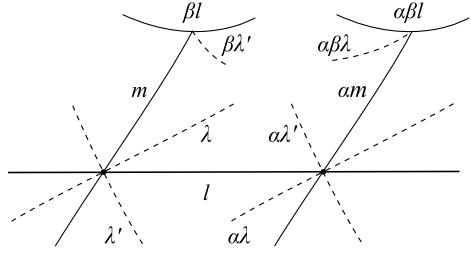}
\caption{The bisectors of two crossing geodesics.}
\label{crossing_geodesics_bisectors}
\end{figure}

Now we start the proof of the main part of Theorem \ref%
{orthogonalprojectionforsingleclosedgeodesic}, which is when $l$ and $m$
cross and their images in $T$ overlap with coherent orientations.

Let $\gamma $ be a closed geodesic on a hyperbolic three-punctured sphere or
once-punctured torus $F$, and let $l$ and $m$ be crossing geodesics in $%
\mathbb{H}^{2}$ above $\gamma $. As in the proof of Lemma \ref%
{disjointgeodesicsbisector}, there is $g$ in $\pi _{1}(M)$ such that $m=gl$,
and $l$ and $m$ project in $T$ to axes of $\alpha $ and $g\alpha g^{-1}$
respectively, which we denote by $A$ and $B$.

Part 2) of Lemma \ref{boundedoverlapbetweenAandgA}, tells us that the
intersection $A\cap B$ is a point or an edge path of $T$ of length at most $%
L(\alpha )-2$, where $L(\alpha )$ denotes the translation length of $\alpha $
acting on $T$. (Note that we cannot have $L(\alpha )=1$, as this would imply
that the closed geodesic $\gamma $ is simple, contradicting the fact that $l$
and $m$ cross.) Thus the intersection $A\cap \alpha B$ is a point or an edge
path of the same length as $A\cap B$ and so is disjoint from $A\cap B$. In
particular, there are two consecutive edges $e$ and $e^{\prime }$ of $A$
whose union meets each of $A\cap B$ and $A\cap \alpha B$ in at most one
vertex, and separates them in $A$. The cutting geodesics $E$ and $E^{\prime
} $ in $\mathbb{H}^{2}$ which correspond to $e$ and $e^{\prime }$ must cross 
$l $, not cross $m$ nor $\alpha m$, and must separate $m$ from $\alpha m$.
In particular, $m$ and $\alpha m$ must be disjoint. But as $m$ and $\alpha m$
cross $l$, it does not immediately follow that there is a geodesic in $%
\mathbb{H}^{2}$ which crosses $l$ orthogonally and separates $m$ from $%
\alpha m$. The geodesics $E$ and $E^{\prime }$ are edges of an ideal
quadrilateral region $Q$ in $\mathbb{H}^{2}$, such that $Q$ is disjoint from 
$m$ and $\alpha m$. We call $Q$ a gap quadrilateral. Note that there may be
several gap quadrilaterals between $m$ and $\alpha m$. The preceding
argument shows only that there must be at least one. Of course, $l$ does
meet $Q$ in some arc. There are two possible configurations. One is that $l$
crosses opposite edges of $Q$, and the other is that $l$ crosses adjacent
edges of $Q$. In the second case, the adjacent edges of $Q$ have a common
vertex $v$, i.e. a point in the circle at infinity of $\mathbb{H}^{2}$. We
will say that $l$ crosses $Q$ at a cusp and that $v$ is the associated
vertex.

Recall that $\pi _{1}(F)$ is a free group of rank $2$, and that $\pi _{1}(F)$
acts freely on the tree $T$ dual to the cutting geodesics in $\mathbb{H}^{2}$
given by the edges of the quadrilaterals. These edges determine a new set of
generators of $\pi _{1}(F)$, which we will denote by $x$ and $y$, and $%
\alpha $ is uniquely expressible as a reduced word in $x$ and $y$ whose
length we denote by $l(\alpha )$. Consider the usual projection of $l$ and $%
m $ to axes $A$ and $B$ in the tree $T$. We let $L(\alpha )$ denote the
length of the cyclically reduced word $W$ conjugate to $\alpha $, so that $%
\alpha $ acts on $A$ by a translation of length $L(\alpha )$. Then the
infinite reduced word $\widetilde{W}$ associated to $A$, the axis of $\alpha 
$ in $T$, is made by concatenating copies of $W$, and each letter of $W$
corresponds to an oriented edge $e$ of $A$. Further each oriented edge of $A$
corresponds to $l$ crossing one of the cutting geodesics which are edges of
our tiling of $\mathbb{H}^{2}$ by ideal quadrilaterals. Let $w$ denote the
subword of $\widetilde{W}$ associated to the interval $A\cap B$. We will
call $w$ the \textit{overlap word of} $l$. Observe that the overlap word is
empty or trivial if $A$ and $B$ do not intersect or if they intersect at a
single point. If $w$ is nontrivial, then we will usually cycle $W$ so that $%
w $ is an initial segment of $W$. The final segment of $W$ will be called
the \textit{gap word of }$l$. Similarly, the conjugate $\alpha ^{g}$ acts on 
$T$\ with axis $B$ on which it acts by a translation of length $L(\alpha )$,
and the infinite reduced word associated to $B$ is also equal to $\widetilde{%
W}$. Let $W^{\prime }$ denote the subword of length $L(\alpha )$ obtained by
reading along $B$ whose initial segment is the word $w$ associated to the
interval $A\cap B$. The final segment of $W^{\prime }$ will be called the 
\textit{gap word of} $m$. Note that the two gap words are usually very
different. In particular, neither can start or end with the same letter as
the other, as that would contradict the fact that $w$ is associated to the
full intersection $A\cap B$.

It will be convenient to introduce the following terminology. We will say
that two reduced words in $x$ and $y$ and their inverses are \textit{%
equivalent} if they are equal or become equal after possibly interchanging $%
x $ and $y$ and/or inverting $x$ or $y$. Note that each of these operations
is an automorphism of the free group $\pi _{1}(F)$ on $x$ and $y$. The
automorphisms of $\pi _{1}(F)\ $generated by these automorphisms will be
called \textit{elementary}. Thus if two words $w$ and $w^{\prime }$ in $x$
and $y$ are equivalent, there is an elementary automorphism of $\pi _{1}(F)$
which sends $w$ to $w^{\prime }$. Such an automorphism simply corresponds to
a re-labelling of the generators of $F$.

At this point, our proof of Theorem \ref%
{orthogonalprojectionforsingleclosedgeodesic} breaks up into several cases,
depending on whether $F$ is a three-punctured sphere or a once-punctured
torus and on the configuration of $l$ in the various gap quadrilaterals. We
will use a combination of arguments in hyperbolic geometry and arguments
with reduced words in a free group.

As the cases of the three-punctured sphere and once-punctured torus are
substantially different, we will devote a separate section to each.

\section{The case of the three-punctured sphere\label%
{section:3-puncturedsphere}}

\begin{lemma}
\label{nocuspcrossingmpliesresultifFissphere}Let $\gamma $ be a closed
geodesic on a hyperbolic three-punctured sphere $F$, and let $l$ and $m$ be
geodesics in $\mathbb{H}^{2}$ above $\gamma $. If there is a gap
quadrilateral $Q$ such that $l$ crosses opposite edges of $Q$, then the
orthogonal projection of $l$ onto $m$ has length strictly less than $%
l(\gamma )$.
\end{lemma}

\begin{proof}
Let $E_{1}$ and $E_{2}$ be the edges of $Q$ crossed by $l$, and let $\lambda 
$ and $\lambda^{\prime }$ be the diagonals of $Q$. So $l$ crosses $\lambda $
and $\lambda^{\prime }$ forming a triangle. As $F$ is a three-punctured
sphere, $\lambda $ and $\lambda^{\prime }$ cross at right angles, and
therefore the perpendicular $\mu$ to $l$ through the point $\lambda \cap
\lambda^{\prime }$ goes through the shaded regions in Figure \ref%
{two_quad_gap}a, and so it lies in the region of $\mathbb{H}^{2}$ bounded by 
$E_{1}$ and $E_{2}$, and separates $E_{1}$ from $E_{2}$. (This type of
argument will be used several times later.) If $\alpha$ is a generator of
the stabilizer of $l$, then $m$ lies between $\mu$ and $\alpha \mu$ or
between $\mu$ and $\alpha^{-1} \mu$, so the orthogonal projection of $m$
onto $l$ is shorter than the translation length of $\alpha $, and so is
shorter than $l(\gamma )$, as required.
\end{proof}

\begin{figure}[ptb]
\centering
\includegraphics[ height=1.2in ]{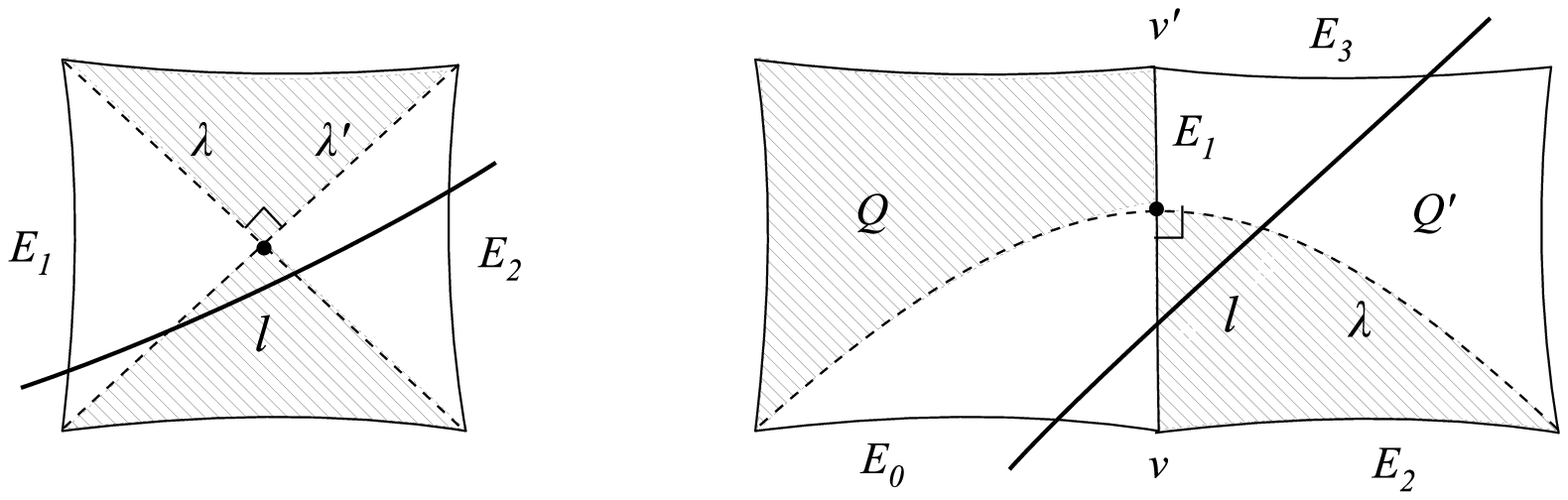}
\caption{{\protect\small a) The geodesic $l$ crossing opposite sides of a
gap quadrilateral. 
b.) The geodesic $l$ crossing two gap quadrilaterals at cusps with different
vertices.}}
\label{two_quad_gap}
\end{figure}

\begin{lemma}
\label{axesinTdonotoverlapinSphere}Let $\gamma $ be a closed geodesic on a
hyperbolic three-punctured sphere $F$, and let $l$ and $m$ be crossing
geodesics in $\mathbb{H}^{2}$ above $\gamma $ whose images in $T$ do not
share an edge. Then the orthogonal projection of $m$ onto $l$ has length
strictly less than $l(\gamma )$.
\end{lemma}

\begin{proof}
As the images of $l$ and $m$ in $T$ do not overlap, $l$ and $m$ cannot cross
together any edges of the quadrilaterals. So $l$ and $m$ cross at an
interior point $p$ of a quadrilateral $Q,$ and $l$ and $m$ cross opposite
edges of $Q$. Let $\alpha $ be a generator of the stabilizer of $l$ and let $%
g$ be an element of $\pi _{1}(F)$ that maps $m$ to $l$, chosen so that $gp$
lies between $p$ and $\alpha p$. Then $gQ$ is a gap quadrilateral for $l$,
and $l$ crosses opposite edges of $gQ$, so by the previous lemma the
orthogonal projection of $m$ to $l$ is shorter than $l(\gamma )$.
\end{proof}

Next we apply this result.

\begin{lemma}
\label{twoquadgapinsphere} Let $\gamma $ be a closed geodesic on a
hyperbolic three-punctured sphere $F$, and let $l$ and $m=gl$ be geodesics
in $\mathbb{H}^{2}$ above $\gamma $. Then either the projection of $m$ onto $%
l$ has length strictly less than $l(\gamma )$, or $l$ crosses all the gap
quadrilaterals at cusps with the same vertex.
\end{lemma}

\begin{proof}
By Lemma \ref{nocuspcrossingmpliesresultifFissphere}, we can assume that $l$
crosses each gap quadrilateral at a cusp. Either all these cusps have the
same associated vertex, or there is a gap quadrilateral $Q$ such that $l$
crosses $Q$ at a cusp with vertex $v$ and crosses the next gap quadrilateral 
$Q^{\prime }$ at a cusp with a different vertex $v^{\prime }$. See Figure %
\ref{two_quad_gap}b. Let $E_{1}$ be the edge that separates $Q$ from $%
Q^{\prime }$, and let $E_{0}$ be the other edge of $Q$ crossed by $l$. Let $%
E_{2}$ be the other edge of $Q^{\prime }$ with one endpoint at $v$, and let $%
E_{3}$ be the other edge of $Q^{\prime }$ crossed by $l$. Thus $m$ and all
its translates by the action of the stabilizer of $l$ lie outside the region
of $\mathbb{H}^{2}$ bounded by $E_{0}$ and $E_{3}$.

As $F$ is a three-punctured sphere, reflection in $E_{1}$ interchanges $%
E_{0} $ and $E_{2}$, so the hyperbolic line $\lambda $ that joins the other
endpoints of $E_{0}$ and $E_{2}$ is perpendicular to $E_{1}$. As $l$ crosses
from $E_{0}$ to $E_{3}$, it must cross the lines $E_{1}$ and $\lambda $
forming a right triangle (see Figure \ref{two_quad_gap}b). Thus the
perpendicular to $l$ from the point $E_{1}\cap \lambda $ crosses $l$ between 
$l\cap E_{1}$ and $l\cap \lambda $, and so it goes through the shaded
regions in the picture. Hence this perpendicular is contained in the region
of $\mathbb{H}^{2}$ bounded by $E_{0}$ and $E_{3}$ and separates $E_{0}$
from $E_{3}$. As $m$ lies on the other side of $E_{0}$, and $\alpha m$ lies
on the other side of $E_{3}$, this perpendicular separates $m$ from $\alpha
m $. Thus, as before, the projection of $m$ to $l$ must be shorter than the
translation length of $\alpha $, which is $l(\gamma )$, as required.
\end{proof}

We are left with the case where $l$ crosses all the gap quadrilaterals at
cusps with the same associated vertex.

\begin{lemma}
\label{basicsaboutgapquadsinsphere} Let $\gamma $ be a closed geodesic on a
hyperbolic three-punctured sphere $F$, and let $l$ and $m=gl$ be crossing
geodesics in $\mathbb{H}^{2}$ above $\gamma $, whose images in $T$ overlap
with coherent orientations. If $l$ crosses all the gap quadrilaterals at
cusps with the same vertex $v$, then $m$ or $\alpha m$ do not cross any cusp
edges at $v$.
\end{lemma}

\begin{proof}
Label the cutting geodesics that end in $v$ by $E_{i}$, $i\in \mathbb{Z}$,
so that $l$ crosses $E_{0},E_{1},\ldots ,E_{n}$. We will call these $E_{i}$%
's cusp edges, and the regions between them cusp regions. Let $Q$ be a gap
quadrilateral, and let $E_{q}$ and $E_{q+1}$ be the edges of $Q$ crossed by $%
l$. As $Q$ separates $m$ and $\alpha m$, it follows that $m$ can only cross $%
E_{i}$'s with $i<q$, and $\alpha m$ can only cross $E_{i}$'s with $i>q+1$,
as shown in Figure \ref{cusps_at_vertex_sphere}a. Note that if $m$ crosses a
cusp edge $E_{i}$ which is also crossed by $l$, then the axes $A$ and $B$ in 
$T$ will share the edge of $T$ corresponding to $E_{i}$. Thus the letter we
read as $l$ crosses $E_{i}$ is part of the overlap word of $l$ and $m$. A\
similar comment holds if $\alpha m$ crosses a cusp edge which is also
crossed by $l$.

Now an orientation for $\gamma $ induces orientations for $l$, $m$ and $%
\alpha m$. As $m$ crosses $l$, the two endpoints $m_{+}$ and $m_{-}$ of $m$
lie on opposite sides of $l$. So $\alpha m$ has endpoints $\alpha m_{+}$ and 
$\alpha m_{-}$, and as $\alpha $ preserves orientation, $m_{+}$ and $\alpha
m_{+}$ lie on one side of $l$ and $m_{-}$ and $\alpha m_{-}$ lie on the
other side. This implies that $m$ and $\alpha m$ "travel around $v$ in
opposite directions", as shown in Figure \ref{cusps_at_vertex_sphere}a.

Now suppose that both $m$ and $\alpha m$ cross at least one $E_{i}$. It
follows that $m$ and $\alpha m$ cross the $E_{i}$'s in opposite directions.
Hence one of $m$ and $\alpha m$ crosses the $E_{i}$'s in the opposite
direction to $l$. Suppose that $m$ does this, as shown in Figure \ref%
{cusps_at_vertex_sphere}a. (If we reverse the orientations of $m$ and $%
\alpha m$ in the figure, then $\alpha m$ will do this.)

Recall that $F\ $has three cusps, and that the simple loops round these
cusps represent $x$, $y$ and $xy$ respectively, when correctly oriented.
Thus by an elementary automorphism of $\pi _{1}(F)$, we can arrange that as $%
l$ crosses $E_{0},E_{1},\ldots ,E_{n}$, either each crossing contributes $x$
to the associated word, or each crossing contributes $x$ and $y$
alternately. In particular, the gap word for $l$ is positive. As the words $%
W $ and $W^{\prime }$ associated to $l$ and $m$ are conjugates and are
reduced, the gap words for $l$ and $m$ have the same abelianisation and are
reduced. In particular, the gap word for $m$ must also be positive. As $m$
crosses the $E_{i}$'s in the opposite direction to $l$, these crossings
yield a negative word, which must therefore be disjoint from the gap word
for $m$. But this implies that these crossings yield part of the overlap
word which is also impossible, as the overlaps of $A$ and $B$ in $T$ are
coherently oriented. This contradiction shows that $m$ cannot cross any cusp
edges at $v$, as required.
\end{proof}

\begin{figure}[ptb]
\centering 
\includegraphics[ height=1.8in]{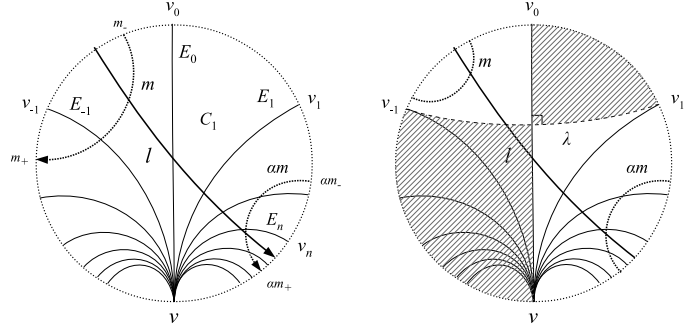}
\caption{{\protect\small The geodesic $l$ crossing the cusps at a vertex.}}
\label{cusps_at_vertex_sphere}
\end{figure}

\begin{lemma}
\label{coherentloopson3-puncturedsphere} Let $\gamma $ be a closed geodesic
on a hyperbolic three-punctured sphere $F$, and let $l$ and $m$ be crossing
geodesics in $\mathbb{H}^{2}$ above $\gamma $ whose images in $T$ overlap
with coherent orientations. If $l$ crosses all the gap quadrilaterals at
cusps with the same vertex $v$, then the orthogonal projection of $l$ onto $%
m $ has length strictly less than $l(\gamma )$.
\end{lemma}

\begin{proof}
Let $E_{0},E_{1},...,E_{n}$ be the cusp edges crossed by $l$. By Lemma \ref%
{basicsaboutgapquadsinsphere}, either $m$ or $\alpha m$ does not cross any $%
E_{i}$'s, so by interchanging the roles of $m$ and $\alpha m$ if necessary,
we can assume that $m$ does not cross any $E_{i}$'s, so that the gap
quadrilaterals start at $E_{0}$. As there is at least one gap quadrilateral, 
$\alpha m$ does not cross any $E_{i}$ with $i\leq 1$. Let $v_{i}$ be the
endpoint of $E_{i}$ other than $v$. Recall that the union of all the cusp
edges that end at $v$ is symmetric under reflection in $E_{0}$. In that
reflection, the point $v_{1}$ is sent to $v_{-1}$. Let $\lambda $ denote the
geodesic joining $v_{1}$ to $v_{-1}$, see Figure \ref{cusps_at_vertex_sphere}%
b. The symmetry implies that $\lambda $ meets $E_{0}$ orthogonally. As $l$
crosses $E_{0}$ and $E_{1}$, it must also cross $\lambda $. As $E_{0}$ and $%
\lambda $ form a right triangle with $l$, the perpendicular to $l$ from the
point $E_{0}\cap \lambda $ crosses $l$ between $l\cap E_{0}$ and $l\cap
\lambda $. So this perpendicular to $l$ has one endpoint between $v_{0}$ and 
$v_{1}$ and the other endpoint between $v$ and $v_{-1}$. As $m$ does not
meet any $E_{i}$, this perpendicular does not cross $m$ nor $\alpha m$ and
it separates $m$ from $\alpha m$. Now the usual argument implies that the
projection of $m$ to $l$ is shorter than $l(\gamma )$.
\end{proof}

Lemmas \ref{disjointgeodesicsbisector}, \ref{oppositelyorientedaxesinT}, \ref%
{axesinTdonotoverlapinSphere}, \ref{twoquadgapinsphere} and \ref%
{coherentloopson3-puncturedsphere} together show that Theorem \ref%
{orthogonalprojectionforsingleclosedgeodesic} holds when $M$ is the
three-punctured sphere.

\section{The case of the once-punctured torus \label%
{section:oncepuncturedtorus}}

In this section, we will proceed much as in the previous one, but the proof
is more delicate. The basic reason for this is that, unlike the
three-punctured sphere, the once-punctured torus is not rigid, but admits a $%
2$--parameter family of complete hyperbolic metrics, and when we cut a
once-punctured torus into a hyperbolic quadrilateral, the resulting
tessellation of $\mathbb{H}^{2}$ may be much less symmetric than the
tessellation from the three-punctured sphere: the only symmetries may be the
covering translations. In particular the diagonals of the quadrilaterals do
not have to intersect at right angles, and the family of cusp edges $%
...E_{-2},E_{-1},E_{0},E_{1},E_{2}...$ that end at a vertex $v$ may not be
invariant under reflections in those edges. We need to discuss what
symmetries still exist.

\begin{figure}[ptb]
\centering 
\includegraphics[ height=1.5in ]{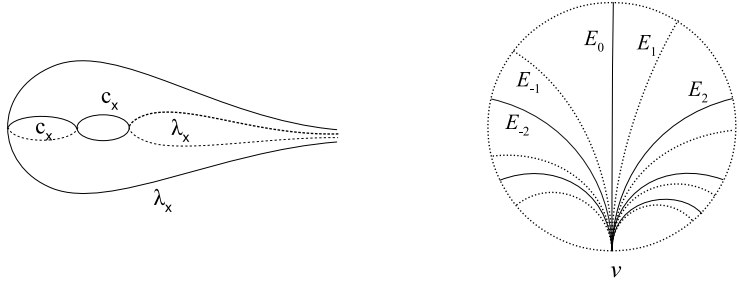}
\caption{{\protect\small Symmetries of the cusp edges for a punctured torus.}
}
\label{symmetries_of_cusp_edges_torus}
\end{figure}

We claim that if one fixes $v$ and restricts attention to just the even
numbered $E_{i}$'s, this family is invariant under reflection in any one of
them, and the same holds for the family of odd numbered edges. To see this,
refer to Figure \ref{symmetries_of_cusp_edges_torus}. The even numbered cusp
edges project to a single cutting geodesic $\lambda _{x}$ in the punctured
torus $F$, and the odd numbered cusp edges project to a single cutting
geodesic $\lambda _{y}$ in $F$. There is a simple closed geodesic $c_{x}$
representing $x$ which meets $\lambda _{x}$ in one point, and there is a
simple closed geodesic $c_{y}$ representing $y$ which meets $\lambda _{y}$
in one point. The punctured torus $F$ admits an orientation preserving
symmetry $\rho $ of order two with three fixed points, the Weierstrass
rotation. One fixed point is $c_{x}\cap c_{y}$, one is $c_{x}\cap \lambda
_{x}$, and one is $c_{y}\cap \lambda _{y}$. Thus each of $c_{x}$, $c_{y}$, $%
\lambda _{x}$ and $\lambda _{y}$ is preserved, but reversed by $\rho $. It
follows that the family of even numbered $E_{i}$'s is invariant under
reflection in any of them, as claimed. Similarly the same holds for the
family of odd numbered edges. Note that if $\lambda $ and $\lambda ^{\prime
} $ have a common endpoint $v$, then the bisector of $\lambda $ and $\lambda
^{\prime }$ must also share that endpoint. It follows from the above that
the reflections in the bisectors of two consecutive $E_{i}$'s preserves the
family of all $E_{i} $'s, interchanging the even numbered $E_{i}$'s with the
odd numbered $E_{i}$'s. Note that these reflections in the $E_{i}$'s and in
the bisectors need not come from a symmetry of the once-punctured torus and
they may not preserve the quadrilaterals of the tessellation.

We will continue to use the terminology introduced at the end of section \ref%
{section:hyperbolic}. By Lemmas \ref{disjointgeodesicsbisector} and \ref%
{oppositelyorientedaxesinT} we are left only with the case when $l$ and $m$
are crossing geodesics whose axes in $T$ do not overlap or overlap with
coherent orientations.

\begin{lemma}
\label{theremustbecuspsifFistorus}Let $\gamma $ be a closed geodesic on a
hyperbolic once-punctured torus $F$, and let $l$ and $m=gl$ be geodesics in $%
\mathbb{H}^{2}$ above $\gamma $ whose images in $T$ overlap with coherent
orientations. Then $l$ crosses at least one gap quadrilateral at a cusp.
\end{lemma}

\begin{proof}
If $l$ does not cross any gap quadrilateral at a cusp then each edge of each
gap quadrilateral crossed by $l$ yields the same letter in $\widetilde{W}$,
say $x$, so that the gap word of $l$ is a positive power of $x$.

Let $w$ denote the overlap subword of $\widetilde{W}$ associated to the
interval $A\cap B$. Then, by cycling $W$ if needed, we can assume that $%
W=wx^{n}$. Reading along the axis $B$ yields that, after cycling, $W^{\prime
}=wz$ for some word $z$ that is the gap word for $m$. As $W^{\prime }$ is
conjugate to $W$, they have the same image in the abelianisation of the free
group $\pi _{1}(F)$. Thus $z=x^{n}$, but then there is no gap. This
contradiction completes the proof of the lemma.
\end{proof}

Next we consider some other special cases.

\begin{lemma}
\label{gapwordisxy} Let $\gamma $ be a closed geodesic on a hyperbolic
once-punctured torus $F$, and let $l$ and $m=gl$ be geodesics in $\mathbb{H}%
^{2}$ above $\gamma $ whose images in $T$ overlap with coherent
orientations. If $l$ has only one gap quadrilateral, then $l$ and $m$ are
disjoint.
\end{lemma}

\begin{proof}
Lemma \ref{theremustbecuspsifFistorus} implies that $l$ crosses the gap
quadrilateral at a cusp. By applying an elementary automorphism of $\pi
_{1}(F)$, if needed, we can assume that the two edges of the gap
quadrilateral crossed by $l$ yield the gap word $xy$ in $\widetilde{W}$. Let 
$w$ denote the overlap subword of $\widetilde{W}$ associated to the interval 
$A\cap B$. Then, by cycling $W$ if needed, we can assume that $W=wxy$.
Reading along the axis $B$ yields $W^{\prime }=wz$, for some word $z$. As $%
W^{\prime }$ is conjugate to $W$, they have the same image in the
abelianisation of the free group $\pi _{1}(F)$. Thus the word $z$ must be $%
xy $ or $yx$. In the first case, $W^{\prime }$ would equal $W$ contradicting
the fact that $w$ is the entire overlap word. It follows that $W^{\prime
}=wyx$. Now Figure \ref{gap_is_positive}a shows that the two ends of $m$
must lie on the same side of $l$, so $l$ and $m$ must be disjoint.
\end{proof}

\begin{figure}[tbp]
\centering 
\includegraphics[ height=1.6in ]{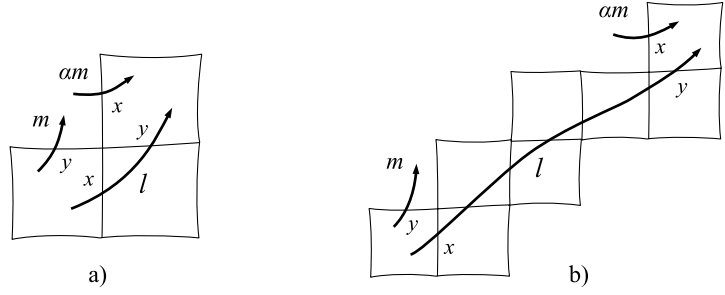}
\caption{{\protect\small If the gap word is positive, starting in $x$ and
ending in $y$.}}
\label{gap_is_positive}
\end{figure}

The argument above can be greatly generalized to obtain the following result.

\begin{lemma}
\label{gapwordequalsxuy}Let $\gamma $ be a closed geodesic on a hyperbolic
once-punctured torus $F$, and let $l$ and $m=gl$ be geodesics in $\mathbb{H}%
^{2}$ above $\gamma $ whose images in $T$ overlap with coherent
orientations. If the gap word of $l$ is a positive word in $x$ and $y$ that
starts with $x$ and ends with $y$, then $l$ and $m$ are disjoint, so the
orthogonal projection of $m$ onto $l$ has length strictly less than $%
l(\gamma )$.
\end{lemma}

\begin{proof}
As in the preceding lemma, if $w$ denotes the overlap subword of $\widetilde{%
W}$ associated to the interval $A\cap B$, then $W=wxuy$ where $u$ is a
positive word, and $W^{\prime }=wz$ for some word $z$. As $W^{\prime }$ is
conjugate to $W$, they have the same image in the abelianisation of the free
group $\pi _{1}(F)$. Thus the word $z$ is also positive. Further $z$ cannot
begin with $x$, nor end with $y$, as either would contradict the fact that $%
w $ is the entire overlap word. It follows that $W^{\prime }=wyu^{\prime }x$%
, for some positive word $u^{\prime }$. Now Figure \ref{gap_is_positive}b
shows that the two ends of $m$ must lie on the same side of $l$, so that $l$
and $m$ must be disjoint. Thus Lemma \ref{disjointgeodesicsbisector} implies
the required result.
\end{proof}

Next we consider some special cases for subwords of gap words.

\begin{lemma}
\label{gapcontainsyxyinverse}Let $\gamma $ be a closed geodesic on a
hyperbolic once-punctured torus, and let $l$ and $m=gl$ be two geodesics in $%
\mathbb{H}^{2}$ above $\gamma $. If the gap word of $l$ contains a subword $%
yxy^{-1}x$ or $yxy^{-1}y^{-1}$ then the projection of $m$ to $l$ is shorter
than $\gamma $.
\end{lemma}

\begin{proof}
By hypothesis $l$ crosses two gap quadrilaterals $Q$ and $Q^{\prime }$ at
cusps with the same vertex $v$ and crosses the next gap quadrilateral $%
Q^{\prime \prime }$ either at a cusp with a different vertex or across
opposite edges. See Figure \ref{gap_contains_yxy-1not}. Let $E_{0}$, $E_{1}$
and $E_{2}$ be the edges of $Q$ and $Q^{\prime }$ crossed by $l$ and with
endpoint $v$, let $E_{3}$ be the other edge of $Q^{\prime \prime }$ with one
endpoint at $v$, and let $E_{4}$ and $E_{5}$ be the other edges of $%
Q^{\prime \prime }$. So $m$ and all its translates by the action of the
stabilizer of $l$ lie outside the region of $\mathbb{H}^{2}$ bounded by $%
E_{0}$, $E_{4}$ and $E_{5}$. Let $\lambda $ be the hyperbolic line that
joins the other endpoints of $E_{0}$ and $E_{3}$.

As the reflection in the bisector of $E_{1}$ and $E_{2}$ interchanges $E_{0}$
and $E_{3}$, the angles that $\lambda $ makes with $E_{1}$ and $E_{2}$
inside $Q^{\prime }$ (marked in the picture by $\epsilon $) are equal, and
so $\epsilon $ is larger than $\pi /2$. As $l$ crosses from $E_{0}$ to $%
E_{4} $ or $E_{5}$, it crosses the lines $E_{2}$ and $\lambda $ forming a
triangle whose other angles must be acute. Therefore the perpendicular to $l$
from the point $E_{2}\cap \lambda $ crosses $l$ between $l\cap \lambda $ and 
$l\cap E_{2}$, so it goes through the shaded regions in the picture. Hence
it lies inside the region of $\mathbb{H}^{2}$ bounded by $E_{0}$, $E_{4}$
and $E_{5}$, and separates $E_{0}$ from $E_{4}$ and $E_{5}$. As $m$ lies on
the other side of $E_{0}$, and $\alpha m$ lies on the other side of $E_{4}$
or $E_{5}$, this perpendicular must separate $m$ from $\alpha m$, so as
before, the projection of $m$ to $l$ must be shorter than the translation
length of $\alpha $, which is $l(\gamma )$.
\end{proof}

\begin{figure}[tbp]
\centering 
\includegraphics[ height=1.3in ]{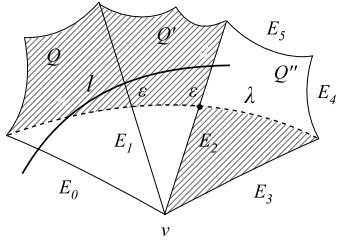}
\caption{{\protect\small If the gap word contains $yxy^{-1}x$ or $%
yxy^{-1}y^{-1}$}}
\label{gap_contains_yxy-1not}
\end{figure}

\begin{lemma}
\label{gapcontainsyx^kyinverse}Let $\gamma $ be a closed geodesic on a
hyperbolic once-punctured torus, and let $l$ and $m=gl$ be two geodesics in $%
\mathbb{H}^{2}$ above $\gamma $. If the gap word of $l$ contains a subword $%
yx^{k}y^{-1}$, with $k>1$, then the projection of $m$ to $l$ is shorter than 
$\gamma $.
\end{lemma}

\begin{proof}
By hypothesis $l$ crosses a gap quadrilateral $Q_{0}$ at a cusp with vertex $%
v$, then crosses $k-1$ gap quadrilaterals $Q_{1},Q_{2},\ldots ,Q_{k-1}$
across opposite edges and then crosses another gap quadrilateral $Q_{k}$ at
a cusp with vertex $v^{\prime }\neq v$, as in Figure \ref%
{gap_contains_yxky-1}.

Let $E_{0}$ and $E_{1}$ be the edges of $Q_{0}$ crossed by $l$, with $E_{1}$
separating $Q_{0}$ from $Q_{1}$, and let $E_{k-1}$ and $E_{k}$ be the edges
of $Q_{k}$ crossed by $l$, with $E_{k-1}$ separating $Q_{k-1}$ from $Q_{k}$.
Let $\tau $ be the covering translation that takes $Q_{0}$ to $Q_{1}$. Note
that $\tau $ must also take $Q_{i-1}$ to $Q_{i}$, for $2\leq i\leq k$. In
particular, $\tau ^{k-1}$ takes $E_{1}$ to $E_{k}$. Let $\lambda $ be the
geodesic joining $\tau ^{-1}v$ to $\tau v$ and let $\lambda ^{\prime }$ be
the geodesic joining $\tau ^{-1}v^{\prime }$ to $\tau v^{\prime }$. Thus $%
\tau ^{k-1}$ also takes $\lambda $ to $\lambda ^{\prime }$. Also $l$ must
cross $\lambda $ and $\lambda ^{\prime }$. Therefore the angles formed by $%
E_{1}$ and $\lambda $, and by $E_{k}$ and $\lambda ^{\prime }$ are equal,
and so the angles marked in Figure \ref{gap_contains_yxky-1} by $\epsilon $
and $\epsilon ^{\prime }$ are supplementary. If $\epsilon \geq \pi /2$, the
perpendicular from the point $\lambda \cap E_{1}$ to $l$ goes through the
shaded regions between $\lambda $ and $E_{1}$, and if $\epsilon ^{\prime
}\geq \pi /2$, the perpendicular from the point $\lambda ^{\prime }\cap
E_{k} $ to $l$ goes through the shaded regions between $\lambda ^{\prime }$
and $E_{k}$. So one of these two perpendiculars lies in the region of $%
\mathbb{H}^{2}$ bounded by $E_{0}$ and $E_{k+1}$ and separates $E_{0}$ from $%
E_{k+1}$, thus separating $m$ from $\alpha m$. As before, this implies that
the projection from $m$ to $l$ must be shorter than $l(\gamma )$.
\end{proof}

\begin{figure}[h]
\centering
\includegraphics[height=0.9in]{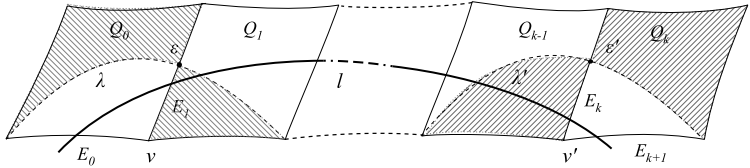}
\caption{{\protect\small If the gap word contains $yx^{k}y^{-1}$, $k>1$.}}
\label{gap_contains_yxky-1}
\end{figure}

\begin{lemma}
\label{gapcontainsx^2andy^2}Let $\gamma $ be a closed geodesic on a
hyperbolic once-punctured torus, and let $l$ and $m=gl$ be two geodesics in $%
\mathbb{H}^{2}$ above $\gamma $. If the gap word of $l$ contains the
subwords $x^{2}$ and $y^{2}$ then the projection of $m$ to $l$ has length
less than $\gamma $.
\end{lemma}

\begin{proof}
If the gap word for $l$ contains the subwords $x^{2}$ and $y^{2}$ then $l$
crosses a gap quadrilateral $Q$ "from left to right" and crosses another gap
quadrilateral $Q^{\prime }$ "from the bottom to the top" crossing both
diagonals of each quadrilateral. See Figure \ref{gap_contains_x2_y2}. There
is a covering translation $\tau $ that takes $Q$ to $Q^{\prime }$ and takes
the diagonals $\lambda $ and $\lambda ^{\prime }$ of $Q$ to the diagonals of 
$Q^{\prime }$, so the angles between their diagonals are equal. Hence either 
$\epsilon \geq \pi /2$ or $\epsilon ^{\prime }\geq \pi /2$. In the first
case, the perpendicular to $l$ from $\lambda \cap \lambda ^{\prime }$ goes
through the shaded regions of $Q$, and in the second case the perpendicular
to $l$ from $\tau \lambda \cap \tau \lambda ^{\prime }$ goes through the
shaded regions of $Q^{\prime }$. So one of these perpendiculars to $l$ lies
in a region of $\mathbb{H}^{2}$ that is not crossed by $m$ or $\alpha m$,
and so it separates $m$ from $\alpha m$. As before, this implies that the
projection of $m$ to $l$ must be shorter than $l(\gamma )$.
\end{proof}

\begin{figure}[ptb]
\centering
\includegraphics[ height=1.6in ]{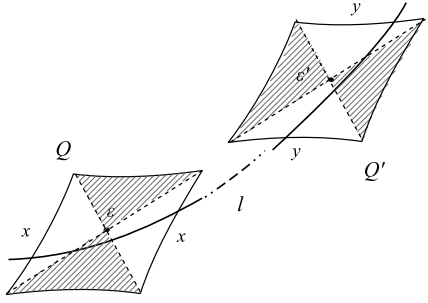}
\caption{{\protect\small If the gap word contains the subwords $x^2$ and $%
y^2 $.}}
\label{gap_contains_x2_y2}
\end{figure}

\begin{lemma}
\label{axesinTdonotoverlapinTorus}Let $\gamma $ be a closed geodesic on a
hyperbolic once-punctured torus, and let $l$ and $m$ be crossing geodesics
in $\mathbb{H}^{2}$ above $\gamma $ whose images $A$ and $B$ in $T$ do not
overlap. Then the orthogonal projection of $m$ onto $l$ has length strictly
less than $l(\gamma )$.
\end{lemma}

\begin{proof}
As the axes $A$ and $B$ in $T$ do not overlap, $l$ and $m$ cross in the
interior of a quadrilateral $Q$, and $l$ and $m$ cross opposite edges of $Q$%
. Therefore the word for $l$ contains a subword $x^{\pm 2}$ and the word for 
$m$ contains a subword $y^{\pm 2}$, or vice versa. As the words for $l$ and $%
m$ are conjugate, each must contain the other subword too. The
quadrilaterals $Q^{\prime }$ and $Q^{\prime \prime }$ crossed by $l$ and $m$
corresponding to these other subwords are gap quadrilaterals for $l$ and $m$
respectively.

Now the proof of the previous lemma shows that either there is a
perpendicular to $l$ in the region of $\mathbb{H}^{2}$ between the edges of $%
Q^{\prime }$ crossed by $l$, or there is a perpendicular to $m$ in the
region of $\mathbb{H}^{2}$ between the edges of $Q^{\prime \prime }$ crossed
by $m$. As before, this implies that the orthogonal projection of $m$ to $l$
is shorter than $l(\gamma )$ or that the orthogonal projection of $m$ to $l$
is shorter than $l(\gamma )$.
\end{proof}

By using the above lemmas, we will show how to reduce to the case when the
gap word has a very special form.

\begin{lemma}
\label{positivewords}Let $\gamma $ be a closed geodesic on a hyperbolic
once-punctured torus $F$, and let $l$ and $m=gl$ be geodesics in $\mathbb{H}%
^{2}$ above $\gamma $ whose images in $T$ overlap with coherent
orientations. By applying elementary automorphisms of $\pi _{1}(F)$, and
replacing $\alpha $ by its inverse if needed, we can arrange that one of the
following cases hold:

\begin{enumerate}
\item the projection of $m$ to $l$ has length less than $\gamma $.

\item the gap word of $l$ is positive and of the form $yx^{d_{1}}yx^{d_{2}}y%
\ldots yx^{d_{r}}y$, or  the form $x^{d_{1}}yx^{d_{2}}y\ldots yx^{d_{r}}$, where each $%
d_{i}\geq 1$.

\item the gap word of $l$ is a subword of $(xyx^{-1}y^{-1})^{N}$, for some $%
N $.
\end{enumerate}
\end{lemma}

\begin{remark}
\label{subwordofpowerofcommutator}In case 3), the geodesic $l$ crosses all
the gap quadrilaterals at cusps with the same associated vertex.
\end{remark}

\begin{proof}
Suppose that the projection of $m$ to $l$ has length greater than that of $%
\gamma $. We will show that part 2) or 3) of the lemma must hold.

Lemma \ref{gapcontainsx^2andy^2} tells us that the gap word $v$ of $l$
cannot contain both $x^{2}$ and $y^{2}$. By applying elementary
automorphisms, it also follows that $v$ cannot contain both $x^{2}$ and $%
y^{-2}$, nor both $x^{-2}$ and $y^{2}$, nor both $x^{-2}$ and $y^{-2}$. It
follows that one of $x$ and $y$ can only occur with exponents $1$ or $-1$.
We will assume that $y$ can only occur with exponents $1$ or $-1$.

Suppose that $x$ has a power which is not $1$ or $-1$. After an elementary
automorphism, we can assume that $v$ contains $x^{k}$, for some $k\geq 2$.
The proof of Lemma \ref{theremustbecuspsifFistorus} shows that $v$ cannot
equal a power of $x$, so by a further elementary automorphism, and perhaps
inversion of $\alpha $, we can assume that $v$ contains the subword $x^{k}y$.

If $v$ contains the subword $x^{k}yx^{l}$, we claim that $l\geq 1$. For if $%
l\leq -1$, the gap word $v$ contains the subword $x^{2}yx^{-1}$, which is
equivalent to the inverse of $yxy^{-2}$. Now Lemma \ref%
{gapcontainsyxyinverse} shows that this cannot occur.

If $v$ contains the subword $x^{k}yx^{l}y^{m}$, we claim that $m=1$. For if $%
l\geq 2$ and $m=-1$, then $v$ contains the subword $yx^{l}y^{-1}$, which is
not possible by Lemma \ref{gapcontainsyx^kyinverse}. And if $l=1$ and $m=-1$%
, then $v$ contains the subword $xyxy^{-1}$ which is equivalent to the
inverse of $yxy^{-1}x$, and so impossible by Lemma \ref%
{gapcontainsyxyinverse}.

If $v$ contains the subword $x^{k}yx^{l}yx^{n}$, we claim that $n\geq 1$. If 
$l\geq 2$, the same argument we used above shows that $n\geq 1$. If $l=1$
and $n\leq -1$, then $v$ contains the subword $yxyx^{-1}$, which is
equivalent to the inverse of $yxy^{-1}x$, and so impossible by Lemma \ref%
{gapcontainsyxyinverse}.

Now a simple inductive argument shows that the entire segment of $v$ which
begins at $x^{k}$ must be positive. By applying the same argument to the
inverse of $\alpha $, we conclude that we can arrange that the entire gap
word $v$ is positive.

Next suppose that $x$ also can only occur with exponents $1$ or $-1$. Lemma %
\ref{gapcontainsyxyinverse} tells us that $v$ cannot contain the subword $%
yxy^{-1}x$, nor any word which is equivalent to $yxy^{-1}x$ or its inverse.
It follows that in the gap word $v$, either all powers of $x$ have the same
sign, and the same holds for all powers of $y$, or that $v$ is a subword of $%
(xyx^{-1}y^{-1})^{N}$, for some $N$. In the first case, we can apply an
elementary automorphism to arrange that $v$ is positive. The second case is
part 3) of the statement of the lemma.

Finally we recall that if the gap word $v$ of $l$ is positive, so is that of 
$m$. It follows from Lemma \ref{gapwordequalsxuy} that $v$ must begin and
end with $x$ or begin and end with $y$, which completes the proof of the
lemma.
\end{proof}

At this point we need to notice that the form of the gap words depends on
our initial choice of generators $x$ and $y$ for $\pi _{1}(F)$. We will now
discuss how to change generators so as to simplify the gap words we are
considering.

The basic operation is to replace one pair of opposite edges of a
quadrilateral by diagonals. We will say that a reduced word in $x$ and $y$
is of \textit{mixed sign}, if it contains positive and negative powers of $x$%
, or if it contains positive and negative powers of $y$. Clearly a reduced
word of mixed sign cannot be equivalent to a positive word.

\begin{lemma}
\label{reduction_of_positive_words}Let $\gamma $ be a closed geodesic on a
hyperbolic once-punctured torus $F$, and let $l$ and $m=gl$ be geodesics in $%
\mathbb{H}^{2}$ above $\gamma $ whose images in $T$ overlap with coherent
orientations. If the gap word of $l$ is $yx^{d_{1}}yx^{d_{2}}y\ldots
yx^{d_{r}}y$, or $x^{d_{1}}yx^{d_{2}}y\ldots yx^{d_{r}}$, where each $%
d_{i}\geq 1$, then the above change of basis yields a new gap word which is
either of mixed sign, or is positive and shorter then the original gap word.
\end{lemma}

\begin{proof}
We start by noting that the gap words for $l$ and $m$ must begin and end in
distinct letters. As we are assuming that the gap word for $l$ is positive,
it follows that the gap word for $m$ is also positive. In particular, each
gap word must begin and end in $x$ or in $y$. As $W\ $and $W^{\prime }$ are
cyclically reduced, it follows that the overlap word $w$ cannot begin or end
in $x^{-1}$ or in $y^{-1}$.

Now replace the sides of the quadrilaterals corresponding to $x$ by the
diagonals shown dotted in Figure \ref{change_of_basis}, keeping the names of
the generators.

If the gap word is $yx^{d_{1}}yx^{d_{2}}y\ldots yx^{d_{r}}y$, and the
overlap word $w$ starts and ends with $x$, then the new gap word is obtained
from the original by reducing each $d_{i}$ by $1.$ See Figure \ref%
{change_of_basis}a. Thus the new gap word is positive and shorter, as
required. If the overlap word starts or ends with $y$, the new gap word is
of mixed sign. See Figure \ref{change_of_basis}b.

\begin{figure}[ptb]
\centering
\includegraphics[ height=3in ]{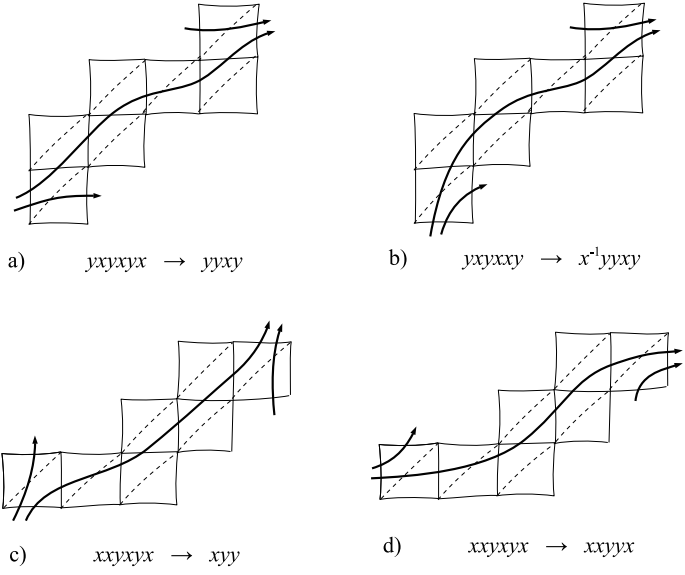}
\caption{{\protect\small A change of basis}}
\label{change_of_basis}
\end{figure}

If the gap word is $x^{d_{1}}yx^{d_{2}}y\ldots yx^{d_{r}}$, then the new gap
word is obtained from the original by reducing each $d_{i}$ by $1$, for $%
1<i<r$, while $d_{1}$ and $d_{r}$ are either reduced by $1$ or remain the
same, depending on how the overlap word starts and ends. See Figures \ref%
{change_of_basis}c,d. Thus the new gap word is positive and is shorter,
except possibly when $r=2$. But this case cannot occur because it would mean
that the gap word $v$ of $l$ equals $x^{d_{1}}yx^{d_{2}}$, so that the gap
word $v^{\prime }$ of $m$ contains only one $y$. Hence $v^{\prime }$ would
have to begin or end in $x$, contradicting the fact that $v$ and $v^{\prime
} $ cannot begin or end with the same letter.
\end{proof}

Now we need to put together everything we have proved so far.

\begin{lemma}
\label{altogether} Let $\gamma $ be a closed geodesic on a hyperbolic
once-punctured torus $F$, and let $l$ and $m=gl$ be geodesics in $\mathbb{H}%
^{2}$ above $\gamma $ whose images in $T$ overlap with coherent
orientations. Then either the projection of $m$ onto $l$ has length strictly
less than $l(\gamma )$, or $l$ crosses all the gap quadrilaterals at cusps
with the same associated vertex.
\end{lemma}

\begin{proof}
Lemma \ref{positivewords} and Remark \ref{subwordofpowerofcommutator}
together tell us that either the projection of $m$ onto $l$ has length
strictly less than $l(\gamma )$, or the gap word of $l$ is positive and of
the form $yx^{d_{1}}yx^{d_{2}}y\ldots yx^{d_{r}}y$, or $x^{d_{1}}yx^{d_{2}}y%
\ldots yx^{d_{r}}$, where each $d_{i}\geq 1$, or $l$ crosses all the gap
quadrilaterals at cusps with the same associated vertex, so it remains to
handle the middle case.

Lemma \ref{reduction_of_positive_words} tells us that in this case, there is
a change of basis which yields a new gap word which is either of mixed sign,
or is positive and shorter then the original gap word. If the new gap word
is positive and one of $x$ and $y$ only occurs with exponent $1$, we can
apply Lemma \ref{reduction_of_positive_words} again. Thus by repeatedly
applying this lemma, we must eventually obtain a gap word of mixed sign or a
positive word which contains both $x^{2}$ and $y^{2}$. In the second case,
Lemma \ref{gapcontainsx^2andy^2} implies that the projection of $m$ onto $l$
has length strictly less than $l(\gamma )$. If the gap word $v$ is of mixed
sign, it cannot be equivalent to a positive word. Thus Lemma \ref%
{positivewords} implies that either the projection of $m$ onto $l$ has
length strictly less than $l(\gamma )$ or that $l$ crosses all the gap
quadrilaterals at cusps with the same associated vertex, thus completing the
proof of the lemma.
\end{proof}

Lemmas \ref{oppositelyorientedaxesinT}, \ref{axesinTdonotoverlapinTorus} and %
\ref{altogether} show that we can reduce to the case when the crossing
geodesics $l$ and $m$ have images in $T$ which overlap with coherent
orientations, and $l$ crosses all the gap quadrilaterals at cusps with the
same vertex $v$. As in the proof of Lemma \ref{basicsaboutgapquadsinsphere},
we label the cusp edges at $v$ by $E_{i}$, $i\in \mathbb{Z}$, so that $%
E_{0},E_{1},...,E_{n}$ are the cusp edges crossed by $l$. See Figure \ref%
{cusps_at_vertex_torus}.

\begin{lemma}
\label{basicsaboutgapquadsintorus}Let $\gamma $ be a closed geodesic on a
hyperbolic once-punctured torus $F$, and let $l$ and $m=gl$ be crossing
geodesics in $\mathbb{H}^{2}$ above $\gamma $ whose images in $T$ overlap
with coherent orientations. If $l$ crosses the $k$ gap quadrilaterals at
cusps with the same vertex $v$, then

1. One of $m$ or $\alpha m$ does not cross any of the cusp edges $%
E_{0},E_{1},...,E_{n}$. And

2. Each of $m$ and $\alpha m$ crosses less than $k-1$ of the remaining cusp
edges.
\end{lemma}

\begin{figure}[ptb]
\centering 
\includegraphics[ height=1.6in]{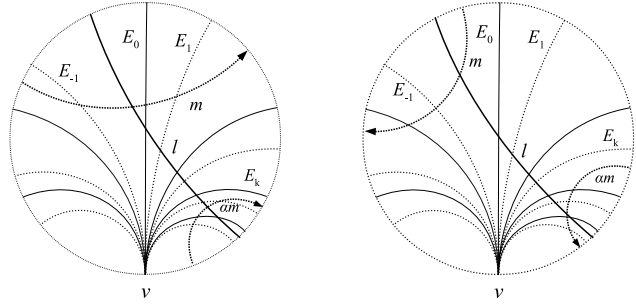}
\caption{{\protect\small The cusps at a vertex for a punctured torus}}
\label{cusps_at_vertex_torus}
\end{figure}

\begin{proof}
1) As in the proof of Lemma \ref{basicsaboutgapquadsinsphere}, let $Q$ be a
gap quadrilateral, and let $E_{q}$ and $E_{q+1}$ be the edges of $Q$ crossed
by $l$. As $Q$ separates $m$ and $\alpha m$, it follows that $m$ can only
cross $E_{i}$'s with $i<q$, and $\alpha m$ can only cross $E_{i}$'s with $%
i>q+1$, as shown in Figure \ref{cusps_at_vertex_torus}a. From that same
proof, we also know that $m$ and $\alpha m$ "travel around $v$ in opposite
directions", as shown in Figure \ref{cusps_at_vertex_torus}a.

Observe that if $m$ crosses two $E_{i}$'s then it crosses all the $E_{j}$'s
between them.

Now suppose that $m$ crosses $E_{0}$ and $E_{-1}$. This implies that the
overlap between $l$ and $m$ starts as they cross $E_{0}$, so that $l$ must
cross all the overlap quadrilaterals and all the gap quadrilaterals at cusps
with vertex $v$. But this implies that $\alpha $ fixes $v$ and so is a
parabolic element, contradicting our assumption that $\gamma $ is a closed
geodesic. We conclude that $m$ cannot cross $E_{0}$ and $E_{-1}$. A similar
argument shows that $\alpha m$ cannot cross $E_{n}$ and $E_{n+1}$. Moreover,
if $m$ crosses $E_{0}$ then $\alpha m$ cannot cross $E_{n}$ and vice versa,
because $m$ and $\alpha m$ cross the cusp edges in opposite directions, but $%
l$ crosses $E_{0}$ and $E_{n}$ in the same direction and $\alpha m$ is a
translate of $m$.

It follows that either $m$ crosses only $E_{i}$'s with $i<0$, or $\alpha m$
crosses only $E_{i}$'s with $i>n$, which proves the first part of the lemma.

2) To prove this, we need only consider $m$, as the roles of $m$ and $\alpha
m$ can be interchanged. If $m$ crosses some $E_{i}$ with $i\geq 0$, then $m$
cannot cross any $E_{i}$'s with $i<0$. Hence $m$ can only cross $E_{i}$'s
with $0\leq i<q$, which proves part 2) of the lemma in this case.

Now suppose that $m$ crosses some $E_{i}$ with $i<0$, so that $m$ cannot
cross any $E_{i}$'s with $i\geq 0$. Then $m$ must cross $E_{-1}$, as
otherwise $E_{-1}$ would separate $m$ from $l$. So the $k$ gap
quadrilaterals must start at $E_{0}$ and end at $E_{k}$, and if $m$ crosses $%
k-1$ cusp edges, they must be $E_{-1},E_{-2},...,E_{-k+1}$.

If $w$ denotes the overlap word for $l$ and $m$, then we can read off the
rest of the word $W$ from the cusp edges that $l$ crosses in the gap, namely 
$E_{0},\ldots ,E_{k}$. As our punctured torus has only one cusp, after an
elementary automorphism of $\pi _{1}(F)$, we can arrange that as $l$ crosses
these cusp edges we read off the letters $%
x,y,x^{-1},y^{-1},x,y,x^{-1},y^{-1},\ldots $. As the images of $l$ and $m$
in $T$ overlap and are oriented coherently, $m$ crosses the cusp edges in
the opposite direction to $l$. Thus as $m$ crosses the cusp edges $%
E_{-1},E_{-2},\ldots ,E_{-k+1}$, we read off the letters $y,x,y^{-1},x^{-1}$
repeatedly. This leads to four cases depending on the size of $k$ modulo $4$.

Case $k=4N$, where $N\geq 1$: Then $W=w(xyx^{-1}y^{-1})^{N}x$. Hence we find
that $W^{\prime }$ has initial segment $w(yxy^{-1}x^{-1})^{N-1}yxy^{-1}$. As 
$l(W^{\prime })=l(W)$, we must have $W^{\prime
}=w(yxy^{-1}x^{-1})^{N-1}yxy^{-1}zu$, where each of $z$ and $u$ denotes one
of $x$ or $y$ or its inverses. As $W^{\prime }$ is a conjugate of $W$,
abelianizing shows immediately that $zu$ has weight $0$ in $x$ and in $y$.
But this is impossible, as $W^{\prime }$, and hence $zu$, is a reduced word.

Case $k=4N+1$, where $N\geq 0$: As in the preceding case, we have that $%
W=w(xyx^{-1}y^{-1})^{N}xy$, and $W^{\prime }$ has initial segment $%
w(yxy^{-1}x^{-1})^{N}$. Therefore $W^{\prime }=w(yxy^{-1}x^{-1})^{N}zu$,
where each of $z$ and $u$ denotes one of $x$ or $y$ or its inverses. As $%
W^{\prime }$ is a conjugate of $W$, abelianizing shows that $zu$ must be
equal to $xy$ or to $yx$. The first case is impossible as $W^{\prime }$ is
reduced, so we must have $zu=yx$. Thus we can write $W$ in the form $wxUy$
and can write $W^{\prime }$ in the form $wyU^{\prime }x$. Now the argument
in the proof of Lemma \ref{gapwordequalsxuy} shows that $l$ and $m$ must be
disjoint, contradicting our assumption that they cross. Note that the use of
Figure \ref{gap_is_positive}b did not depend on the gap word $xUy$ being
positive. That hypothesis was used in the proof of Lemma \ref%
{gapwordequalsxuy} to show that the gap word for $m$ must be of the form $%
yU^{\prime }x$, but in the present situation, that is given.

Case $k=4N+2$, where $N\geq 0$: As in the preceding case, we have $%
W=w(xyx^{-1}y^{-1})^{N}xyx^{-1}$, and $W^{\prime }$ has initial segment $%
w(yxy^{-1}x^{-1})^{N}y$. Thus $W^{\prime }=w(yxy^{-1}x^{-1})^{N}yzu$, where
each of $z$ and $u$ denotes one of $x$ or $y$ or its inverses. As $W^{\prime
}$ is a conjugate of $W$, abelianizing shows immediately that $zu$ has
weight $0$ in $x$ and in $y$. But this is impossible, as $W^{\prime }$, and
hence $zu$, is a reduced word.

Case $k=4N+3$, where $N\geq 0$: As in the preceding case, we have $%
W=w(xyx^{-1}y^{-1})^{N}xyx^{-1}y^{-1}$, and $W^{\prime }$ has initial
segment $w(yxy^{-1}x^{-1})^{N}yx$. Thus $W^{\prime
}=w(yxy^{-1}x^{-1})^{N}yxzu$, where each of $z$ and $u$ denotes one of $x$
or $y$ or its inverses. As $W^{\prime }$ is a conjugate of $W$, abelianizing
shows that $zu$ must be equal to $x^{-1}y^{-1}$ or to $y^{-1}x^{-1}$. The
first case is impossible as $W^{\prime }$ is reduced, so we must have $%
zu=y^{-1}x^{-1}$. Thus we can write $W$ in the form $wxUy^{-1}$ and can
write $W^{\prime }$ in the form $wyU^{\prime }x^{-1}$. Now a similar
argument to that in the proof of Lemma \ref{gapwordequalsxuy} shows that $l$
and $m$ must be disjoint, contradicting our assumption that they cross. Note
that in this case, we need a modified version of Figure \ref{gap_is_positive}%
b. The above contradictions, for any value of $k$, complete the proof of
part 2).
\end{proof}

\begin{figure}[ptb]
\centering 
\includegraphics[ height=1.8in]{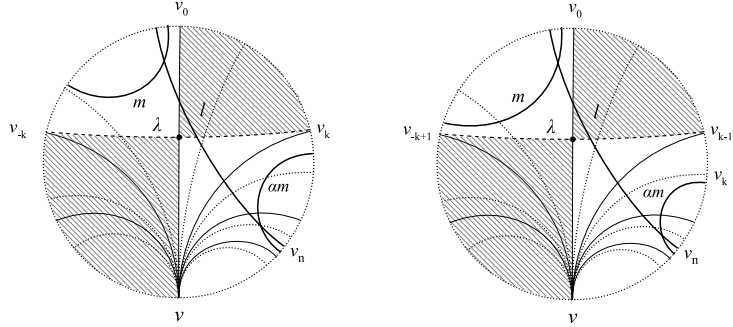}
\caption{ a) $k$ even. \hskip 90pt b) $k$ odd. \hfill}
\label{cusps_at_vertex_torus_2}
\end{figure}

\begin{lemma}
\label{coherentloopsonpuncturedtorus} Let $\gamma $ be a closed geodesic on
a hyperbolic once-punctured torus $F$, and let $l$ and $m=gl$ be geodesics
in $\mathbb{H}^{2}$ above $\gamma $ whose images in $T$ overlap with
coherent orientations. If $l$ crosses all the gap quadrilaterals at cusps
with the same vertex, then the orthogonal projection of $l$ onto $m$ has
length strictly less than $l(\gamma )$.
\end{lemma}

\begin{proof}
Let $E_{0},E_{1},...,E_{n}$ be the cusp edges crossed by $l$. By Lemma \ref%
{basicsaboutgapquadsintorus} we can assume that the $k$ gap quadrilaterals
start at $E_{0}$ and end at $E_{k}$, so $\alpha m$ can only cross $E_{i}$'s
with $i>k$, and $m$ can only cross $E_{i}$'s with $-(k-1)<i<0$. Recall that
the union of all the even cusp edges that end at $v$ is symmetric under
reflection in $E_{0}$. If $v_{i}$ is the endpoint of the cusp edge $E_{i}$,
this reflection sends the point $v_{2i}$ to $v_{-2i}$. See Figure \ref%
{cusps_at_vertex_torus_2}.

If $k$ is even, we let $\lambda $ denote the geodesic joining $v_{k}$ to its
reflected image $v_{-k}$. The symmetry implies that $\lambda $ meets $E_{0}$
orthogonally. As $l$ crosses the cusp geodesic $E_{k}$, it must also cross $%
\lambda $. Let $\mu $ denote the perpendicular to $l$ from the point $%
E_{0}\cap \lambda $. Clearly one endpoint of $\mu $ lies between $v_{0}$ and 
$v_{k}$, and the other endpoint lies between $v$ and $v_{-k}$. As $m$ cannot
cross $E_{-k}$, it follows immediately that $\mu $ cannot meet $m$ nor $%
\alpha m$. As usual, this implies that the orthogonal projection of $l$ onto 
$m$ has length strictly less than $l(\gamma )$.

If $k$ is odd, we apply the above paragraph using the even number $k-1$ in
place of $k$. As $m$ cannot cross $E_{-k+1}$, it again follows immediately
that $\mu $ cannot meet $m$ nor $\alpha m$, which again implies that the
orthogonal projection of $l$ onto $m$ has length strictly less than $%
l(\gamma )$. Thus the result follows in either case, as required.
\end{proof}

Lemmas \ref{disjointgeodesicsbisector}, \ref{oppositelyorientedaxesinT}, \ref%
{axesinTdonotoverlapinTorus}, \ref{altogether} and \ref%
{coherentloopsonpuncturedtorus} together show that Theorem \ref%
{orthogonalprojectionforsingleclosedgeodesic} holds when $M$ is the
three-punctured torus.

\section{The Main result\label{section:Mainresults}}

At this point we are ready to complete the proof of our main result.

\begin{theorem}
\label{orthogonalprojectionforsingleclosedgeodesic}Let $\gamma $ be a closed
geodesic on an orientable hyperbolic surface $M$ and let $l$ and $m$ be
distinct geodesics in $\mathbb{H}^{2}$ above $\gamma $. Then the orthogonal
projection of $l$ onto $m$ has length strictly less than $l(\gamma )$.
\end{theorem}

\begin{proof}
The results in the preceding three sections yield a proof of the theorem in
the special cases when $M$ is finitely covered by a three-punctured sphere
or by a once-punctured torus. In general, as discussed at the start of
section \ref{section:hyperbolic}, we need to consider a cover of $M$ which
is a surface $F$ homeomorphic to a sphere with three discs removed or to a
torus with one disc removed, but the hyperbolic metric on $F$ need not have
finite volume. Thus the once-punctured torus may be replaced by a surface
with no cusp, and a closed geodesic as the boundary of its convex core, and
the three-punctured sphere may be replaced by a surface with less than three
cusps, and some closed geodesics as the boundary of the convex core. The
crucial step which allows one to proceed in the same way as in the preceding
sections is to choose the cutting geodesics in $F$ to be orthogonal to any
closed geodesic boundary components of the convex core of $F$. Again this
yields a tiling of $\mathbb{H}^{2}$ by quadrilaterals, but these
quadrilaterals may be ultra ideal, i.e. have vertices beyond infinity. The
diagonals of an ultra ideal quadrilateral join opposite ends, and are
orthogonal to any closed geodesic boundary components of the convex core of
the quadrilateral.

Recall that the hyperbolic three-punctured sphere is the double of an ideal
triangle. Similarly, a hyperbolic three-holed sphere $F$ is the double of a
triangle, some of whose vertices are ultra ideal. We choose two of the
common edges of these triangles to be the cutting geodesics for $F$, so they
cut $F$ into a quadrilateral $Q$ with some ultra ideal vertices which admits
a reflectional symmetry in a diagonal. In particular, the diagonals for $Q$
meet at right angles. The convex core of a one-holed torus may not admit any
reflectional symmetries.

Now all the lemmas in sections \ref{section:hyperbolic}, \ref%
{section:3-puncturedsphere} and \ref{section:oncepuncturedtorus} can be
proved in essentially the same way, but the references to cusps of the
quadrilaterals will need to be replaced by references to the vertices of the
quadrilaterals. References to the cusp edges, which are geodesics with one
end at the cusp, will need to be replaced by references to edges which go to
the same vertex of the tiling of the universal covering of $F$. In the case
of a three-holed sphere, the symmetries of the quadrilaterals show that the
set of edges of the tiling with a common vertex is invariant under
reflections in any of them. In the case of a one-holed torus, it has a
rotational symmetry of order two. Thus, as discussed at the start of section %
\ref{section:oncepuncturedtorus}, the cutting geodesics for the tiling of $%
\mathbb{H}^{2}$ by quadrilaterals admit some limited symmetries. Namely if
one considers the family $\ldots E_{-2},E_{-1},E_{0},E_{1},E_{2},\ldots $ of
those geodesics that end at a vertex of the tiling, the even numbered ones
are invariant under reflection in any of them, and the same holds for the
odd numbered ones. Further the entire family of $E_{i}$'s is invariant under
reflection in the bisector of two consecutive $E_{i}$'s. Finally recall that
some of our earlier arguments depended on the fact that the element $\alpha $
of $\pi _{1}(F)$ carried by the closed geodesic $\gamma $ is not a parabolic
element. We weil also need the fact that $\gamma $ cannot be a component of
the convex core of $F$. For then $\gamma $ would be simple, so that two
distinct geodesics in $\mathbb{H}^{2}$ which lie above $\gamma $ cannot
cross.
\end{proof}

We can now deduce the following bounds on the self-intersection angles of a
hyperbolic geodesic in terms of its length.

\begin{corollary}
\label{bounds for angles} If $\gamma $ is a closed oriented geodesic on an
orientable hyperbolic surface, and $\phi $ is the angle formed by the two
outgoing arcs of $\gamma $ at a self-intersection point, then $\Pi (\frac{%
l(\gamma )}{2})<\phi <\pi -\Pi (\frac{l(\gamma )}{4})$.
\end{corollary}

\begin{proof}
The first inequality follows immediately from Theorem \ref%
{orthogonalprojectionforsingleclosedgeodesic}. The second inequality follows
by applying Lemma \ref{oppositelyorientedprojection} to the supplementary
angle $\pi -\phi $.
\end{proof}

We can also deduce the following bounds on the lengths of polygons formed by
the lines above $\gamma$ in $\mathbb{H}^2$.

\begin{corollary}
\label{boundsforpolygons}If $\gamma $ is a closed geodesic on an orientable
hyperbolic surface $M$, then the triangles formed by the geodesic lines
above $\gamma $ in $\mathbb{H}^{2}$ have sides shorter than $l(\gamma )$,
and the $n$--gons have sides shorter than $(n-2)l(\gamma )$.
\end{corollary}

\begin{proof}
The length of a side $s$ of a closed polygon $P$ in $\mathbb{H}^{2}$ is
bounded above by the sum of the lengths of the orthogonal projections of the
other sides to the line containing $s$. These lengths are bounded by the
lengths of the projections of the lines containing them to $s$, and in the
case of the two sides adjacent to $s$, by half of that length.
\end{proof}

If the polygon $P$ has oriented boundary (with the orientations of the sides
induced by an orientation of $\gamma $) then Lemmas \ref%
{disjointoppositelyorientedprojection} and \ref{oppositelyorientedprojection}
show that the sides of $P$ are shorter than $\frac{n-1}{2}l(\gamma )$.

Now we can deduce the following result.

\begin{theorem}
\label{orthogonalprojectionfortwoclosedgeodesics}Let $\gamma $ and $\delta $
be closed geodesics on an orientable hyperbolic surface $M$, and let $l$ and 
$m$ be distinct geodesics in $\mathbb{H}^{2}$ above $\gamma $ and $\delta $
respectively. Then the orthogonal projection of $l$ onto $m$ has length
strictly less than $l (\gamma ) +l (\delta )$.
\end{theorem}

Interestingly, the case of two intersecting lines above distinct geodesics
reduces to the case of two disjoint lines above a single geodesic, while the
case of two disjoint lines above distinct geodesics reduces to the case of
two intersecting lines above a single geodesic. First we consider the
situation where $l$ and $m$ are disjoint.

\begin{figure}[ptb]
\centering 
\includegraphics[
height=1.5in]{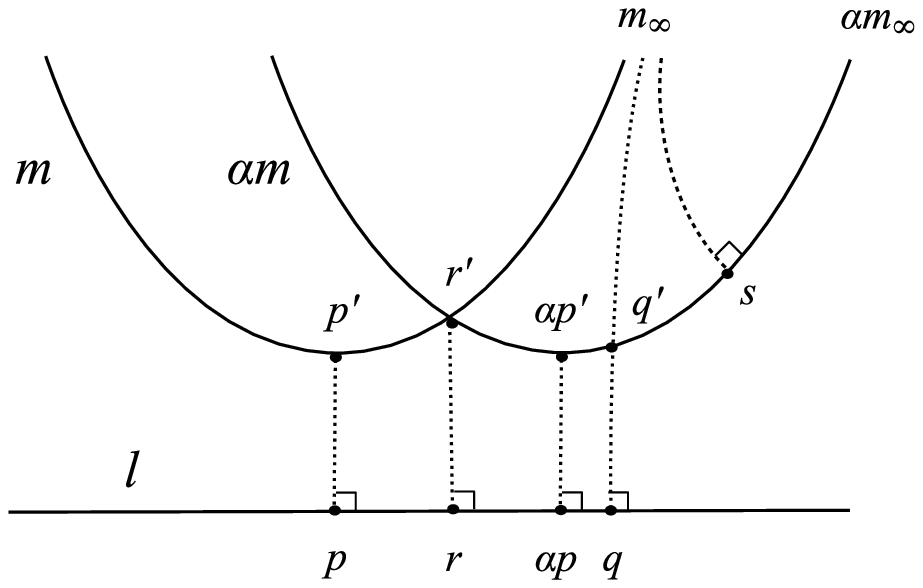}
\caption{{\protect\small Orthogonal projection of geodesics above different
curves}}
\label{orthogonal_projection_for_two_disjoint_geodesics}
\end{figure}

\begin{lemma}
Let $\gamma $ and $\delta $ be closed geodesics on an orientable hyperbolic
surface $M$, and let $l$ and $m$ be disjoint geodesics in $\mathbb{H}^{2}$
above $\gamma $ and $\delta $ respectively. Then the orthogonal projection
of $l$ onto $m$ has length strictly less than $l(\gamma )+l(\delta )$.
\end{lemma}

\begin{proof}
Without loss of generality, we can assume that $l(\gamma )\leq l(\delta )$.
As usual we let $\alpha $ denote the element of $\pi _{1}(F)$ represented by 
$\gamma $, so that $\alpha $ acts on $\mathbb{H}^{2}$ with $l$ as its axis.
If $m$ and $\alpha m$ are disjoint, the proof of Lemma \ref%
{disjointgeodesicsbisector} shows that the orthogonal projection of $m$ onto 
$l$ has length strictly less than $l(\gamma )$, thus proving the lemma in
this case. Now suppose that $m$ and $\alpha m$ cross at a point $r^{\prime }$%
. Let $p$ and $p^{\prime }$ be the closest points of $l$ and $m$, so the arc 
$[p,p^{\prime }]$ is perpendicular to $l$ and to $m$. Let $m_{\infty }$ be
the point at infinity of $m$ such that $r^{\prime }$ lies in the ray $%
(p^{\prime },m_{\infty })$, let $q$ be the foot of the perpendicular to $l$
from $m_{\infty }$, let $q^{\prime }$ be the intersection of $qm_{\infty }$
with $\alpha m$, and let $s$ be the foot of the perpendicular to $\alpha m$
from $m_{\infty }$. See Figure \ref%
{orthogonal_projection_for_two_disjoint_geodesics}. To prove the lemma we
need to show that the arc $[p,q]$, which is half of the orthogonal
projection of $m$ to $l$, is shorter than half of $l(\gamma )+l(\delta )$.

If $q$ lies between $p$ and $\alpha p$ then $l(p,q)<l(p,\alpha p)=l(\gamma
)\leq \frac{1}{2}(l(\gamma )+l(\delta ))$, as required. Next suppose that $q$
lies beyond $\alpha p$, as in Figure \ref%
{orthogonal_projection_for_two_disjoint_geodesics}. Then the angle $%
m_{\infty }q^{\prime }\alpha m_{\infty }$ must be acute because it is equal
to an interior angle of the geodesic quadrilateral $q^{\prime }q\alpha
p\alpha p^{\prime }$ whose other interior angles are right angles. It
follows that $s$ must lie on the ray $[q^{\prime },\alpha m_{\infty })$, as
shown in Figure \ref{orthogonal_projection_for_two_disjoint_geodesics}.
Therefore $l(r,q)<l(r^{\prime },q^{\prime })<l(r^{\prime },s)$. As $%
(r^{\prime },s)$ is half of the orthogonal projection of $m$ to $\alpha m$,
Theorem \ref{orthogonalprojectionforsingleclosedgeodesic} tells us that $%
l(r^{\prime },s)<\frac{1}{2}l(\delta )$. Thus $l(r,q)<\frac{1}{2}l(\delta )$%
. As $l(p,r)=\frac{1}{2}l(p,\alpha p)=\frac{1}{2}l(\gamma )$, it follows
that $l(p,q)=l(p,r)+l(r,q)<\frac{1}{2}l(\gamma )+\frac{1}{2}l(\delta )$, as
required.

If $q=\alpha p$ then $q^{\prime }=\alpha p^{\prime }=s$, so $%
l(p,q)=l(p,r)+l(r,\alpha p)<l(p,r)+l(r^{\prime },\alpha p^{\prime
})=l(p,r)+l(r^{\prime },s)$. As in the preceding case, $l(r^{\prime },s)<%
\frac{1}{2}l(\delta )$, by Theorem \ref%
{orthogonalprojectionforsingleclosedgeodesic}, and $l(p,r)=\frac{1}{2}%
l(\gamma )$. Again it follows that $l(p,q)<\frac{1}{2}l(\gamma )+\frac{1}{2}%
l(\delta )$, as required.
\end{proof}

We now consider the case when $l$ and $m$ cross. In this case we get a
better bound for the projection length, which depends only on Lemma \ref%
{disjointgeodesicsbisector}.

\begin{lemma}
\label{orthogonalprojectionfortwoclosedcrossinggeodesics}Let $\gamma $ and $%
\delta $ be closed geodesics on an orientable hyperbolic surface $M$, and
let $l$ and $m$ be distinct crossing geodesics in $\mathbb{H}^{2}$ above $%
\gamma $ and $\delta $ respectively. Then the orthogonal projection of a
bisector of $l$ and $m$ onto $m$ has length strictly less than $l(\gamma
)+l(\delta )$.
\end{lemma}

\begin{proof}
As $l$ and $m$ cross, so do $\gamma $ and $\delta$. The idea of our proof is
to perform cut and paste on $\gamma $ and $\delta $ so as to obtain a new
curve of length $l (\gamma ) +l (\delta )$. The shortest closed geodesic $%
\sigma $ in the homotopy class of this curve must have $l (\sigma ) <l
(\gamma ) +l (\delta )$.

\begin{figure}[ptb]
\centering 
\includegraphics[
height=1.4in]{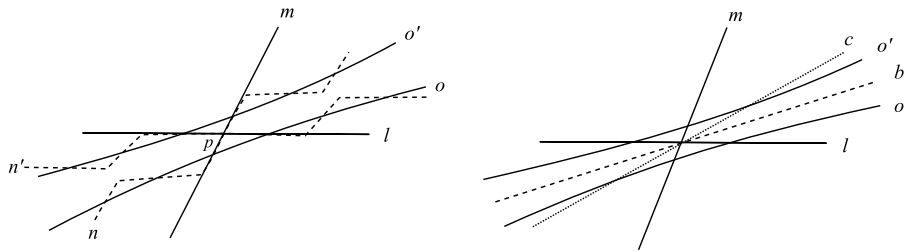}
\caption{The geodesics $o$ and $o^{\prime }$ arising from piecewise
geodesics $n$ and $n^{\prime }$.}
\label{orthogonal_projection_for_two_crossing_geodesics}
\end{figure}

Consider the crossing geodesics $l$ and $m$ in $\mathbb{H}^{2}$, and let $p$
denote the intersection point $l\cap m$. A cut and paste at $p$ determines a
cut and paste operation on $\gamma $ and $\delta $ at a single point, which
must yield a single piecewise geodesic closed curve $\eta $ with two corners
at the cut and paste point. Thus there are two piecewise geodesic paths $n$
and $n^{\prime }$ in $\mathbb{H}^{2}$ above $\eta $, which pass through $p$.
See Figure \ref{orthogonal_projection_for_two_crossing_geodesics}. Each
proceeds along $l$ from $p$ for a distance equal to $l(\gamma )$, then turns
a corner onto a translate of $m$, and proceeds a distance $l(\delta )$, etc.
We let $\sigma $ denote the closed geodesic in the homotopy class of $\eta $%
. Corresponding to $n$, we can construct a geodesic $o$ in $\mathbb{H}^{2}$
above $\sigma $ by simply joining the midpoints of the geodesic segments of $%
n$ by geodesic segments. And similarly we can construct a geodesic $%
o^{\prime }$ in $\mathbb{H}^{2}$ above $\sigma $ by simply joining the
midpoints of the geodesic segments of $n^{\prime }$. The reason why $o$ and $%
o^{\prime }$ are geodesic rather than just piecewise geodesic, is that $n$,
and hence $o$, is invariant under rotation through $\pi $ about each of the
points where $n$ meets $o$, and a similar statement holds for $n^{\prime }$
and $o^{\prime }$. As $n$ and $o$ have the same stabilizer, and $n^{\prime }$
and $o^{\prime }$ similarly, it follows that $o$ and $o^{\prime }$ must lie
above the closed geodesic $\sigma $. As $o^{\prime }$ is a translate of $o$
by the action of the stabilizer of $l$, it follows that $o$ and $o^{\prime }$
are disjoint. As a rotation through $\pi $ about $p$ sends $o$ to $o^{\prime
}$, their bisector $b$ goes through the point $p$.

Now Lemma \ref{disjointgeodesicsbisector} says that the orthogonal
projections of $b$ to $o$ and $o^{\prime }$ are not larger that $l(\sigma )$%
. These have the same length as the projections of $o $ and $o^{\prime }$ to 
$b$, which are the same interval of $b$. This interval contains the feet of
the perpendiculars to $b$ from the endpoints of $l$, so the projection of $l$
to $b$ is shorter than the projections of $o $ and $o^{\prime }$ to $b$. For
the same reason, the projection of $m$ to $b$ is shorter than the
projections of $o$ and $o^{\prime }$ to $b$. Thus the orthogonal projections
of $b$ to $l$ and $m$ are shorter than $l(\sigma )$.

Let $c$ denote the bisector of $l$ and $m$. Then one of $l$ and $m$ forms a
larger angle with $c$ than with $b$, so that the orthogonal projection of $c$
to one of $l$ and $m$ is shorter than the orthogonal projection of $b$.
Hence the orthogonal projections of $c$ to $l$ and $m$, which are equal,
must be shorter than $l(\sigma )$, which is smaller than $l(\gamma
)+l(\delta )$, as required.
\end{proof}

Lemma \ref{orthogonalprojectionfortwoclosedcrossinggeodesics} gives a lower
bound for the intersection angles of two hyperbolic geodesics.

\begin{corollary}
\label{angleoftwodistinctgeodesics} The intersection angles of two closed
geodesics $\gamma$ and $\delta$ in an orientable hyperbolic surface are
larger than $2 \Pi(\frac{l(\gamma)+l(\delta)}{2})$.
\end{corollary}

Theorem \ref{orthogonalprojectionfortwoclosedgeodesics} also shows that the
sides of the $n$--gons formed by the lines above a family of geodesics $%
\gamma _{1},\gamma _{2},\ldots ,\gamma _{k}$ are shorter than $2(n-2)\max
\{l(\gamma _{i})\}$.

\ 

The bound for the lengths of the orthogonal projections given in Theorem \ref%
{orthogonalprojectionforsingleclosedgeodesic} is optimal, in the following
sense:

\begin{claim}
For each hyperbolic surface $M$ with $\chi (M)<0$, there is a sequence of
closed geodesics $\gamma _{n}$ and lines $l_{n}$, $m_{n}$ in $\widetilde{F}$
above $\gamma _{n}$ so that the length of the orthogonal projection of $%
m_{n} $ to $l_{n}$, divided by the length of $\gamma _{n}$, converges to $1$.
\end{claim}

\begin{figure}[ptb]
\centering
\includegraphics[height=0.8in]{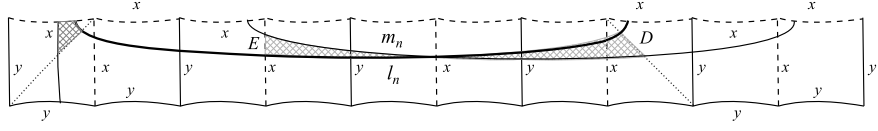}
\caption{Two lines with long projections}
\label{bounds_are_sharp}
\end{figure}

\begin{proof}
As each hyperbolic surface $M$ with $\chi (M)<0$ has a (usually infinite)
covering which is a hyperbolic sphere with three holes, it suffices to prove
the claim when $M$ is a hyperbolic sphere with three holes, each
corresponding to a cusp or a boundary curve of its convex core. We can cut $%
M $ along two infinite geodesics meeting the boundary curves of the convex
core orthogonally to get an ideal or ultra-ideal quadrilateral. In each
case, this quadrilateral has a reflectional symmetry that shows that the
diagonals intersect at right angles and the distances between opposite sides
of the quadrilateral are equal. Let $x$ and $y$ be generators of $\pi
_{1}(M) $ dual to the cutting geodesics. To get a sequence $\gamma _{n}$
where $l_{n} $ intersects $m_{n}$, let $\gamma _{n}$ be the geodesic
represented by the cyclic word $(xy)^{n}x$. Let $l_{n}$ and $m_{n}$ be two
pre-images of $\gamma _{i}$ in the universal cover $\widetilde{M}$ of $M$ as
in Figure \ref{bounds_are_sharp}. Since the infinite words corresponding to $%
l_{n}$ and $m_{n}$ overlap in a word $(xy)^{n-1}x$, then $l_{n}$ and $m_{n}$
cross together $2(n-1)$ quadrilaterals.

We claim that the arc $a_{n}$ where $l_{n}$ crosses these quadrilaterals is
contained in the projection of $m_{n}$ to $l_{n}$. To see this, consider the
shaded quadrangle in Figure \ref{bounds_are_sharp}, formed by $l_{n}$, $%
m_{n} $, the first cutting geodesic $E$ that they cross together and the
last diagonal $D$ that they cross together. Recall that $M$ is the double of
an ideal or ultra ideal triangle, and so it admits a reflection symmetry
which interchanges these triangles. It follows that in the universal cover $%
\widetilde{M}$ of $M$, reflection in any quadrilateral edge or diagonal is a
symmetry of the tiling of $\widetilde{M}$ by quadrilaterals. In particular,
a reflection in $D$ preserves the tiling and preserves, but inverts, the
geodesic $l$, because it inverts the infinite word which is the unwrapping
of $(xy)^{n}x$. It follows that the angle at the point where $l_{n}$ crosses 
$D$ is $\pi /2$. Also, the reflection of $l_{n}$ in $E$ is a geodesic that
crosses two of the $2(n-1)$ quadrilaterals crossed by $l_{n}$ and $m_{n}$
and "turns north" at the third quadrilateral. So the reflection of $l_{n}$
lies "north" of $l_{n}$ in these quadrilaterals, and therefore the shaded
angle at the intersection of $l_{n}$ and $E$ is greater than its
supplementary angle, and so it is greater than $\pi /2$.

Now, if $d$ is the distance between opposite sides of the quadrilaterals
then $\gamma _{n}$ is shorter than $(2n+1)d$, as it is homotopic to a
polygonal curve that runs through the middle arcs of the quadrilaterals,
made of $2n$ arcs of length $d$ and $2 $ subarcs of length $d/2$. On the
other hand, the orthogonal projection of $m_{n}$ to $l_{n}$ is longer than
the arc $a_{n}$, which has length at least $2(n-1)d$. So the ratio between
the length of the projection and the length of the geodesic is larger than $%
\frac{2n-2}{2n+1}$, so it converges to $1$ as $n$ goes to infinity.

To get a sequence $\gamma _{n}$ where $l_{n}$ is disjoint from $m_{n}$, let $%
\gamma _{n}$ be the geodesic represented by the word $x(xy)^{n}y$, and let $%
l_{n}$ and $m_{n}$ be two pre-images of $\gamma _{i}$ in $\widetilde{F}$ so
that the infinite words corresponding to $l_{n}$ and $m_{n}$ overlap in the
word $(xy)^{n-1}$. So $l_{n}$ and $m_{n}$ cross together $2(n-1)-1$
quadrilaterals. One can show as before that the arc where $l_{n}$ crosses
these quadrilaterals is contained in the projection of $m_{n}$ to $l_{n}$.
So $\gamma _{n}$ is shorter than $2(n+1)d$, while the orthogonal projection
of $m_{n}$ to $m_{n}$ is longer than $(2n-3)d$. Therefore the ratio between
the length of the projection and the length of $\gamma _{n}$ is larger than $%
\frac{2n-3}{2n+2}$, and so it converges to $1$ as $n$ goes to infinity.
\end{proof}

The upper bound for the lengths of the orthogonal projections of lines above
two different geodesics given in Theorem \ref%
{orthogonalprojectionfortwoclosedgeodesics} is also optimal. Consider the
sequence of geodesics $\gamma _{n}$ and $\delta _{n}$ represented by the
words $(xy)^{n}x$ and $(xy)^{n+2}x$ in a three-holed sphere. Take two
geodesic lines $l_{n}$ and $m_{n}$ above $\gamma _{n}$ and $\delta _{n}$
whose axes have a common subword $(xy)^{n}x(xy)^{n}$. Then the projection of 
$m_{n}$ to $l_{n}$ has length greater than $4nd$, while the lengths of $%
\gamma _{n}$ and $\delta _{n}$ are smaller than $(2n+1)d$ and $(2n+5)d$
respectively so the ratio between the length of the projection and the sum
of the lengths of $\gamma _{n}$ and $\delta _{n}$ is larger than $\frac{4n+1%
}{4n+6}$, and so it converges to $1$ as $n$ goes to infinity.

Although the bound for the projection length given in Theorem \ref%
{orthogonalprojectionforsingleclosedgeodesic} is optimal in relative terms,
it may not be so in absolute terms. The above examples suggest that the
projection length for two intersecting geodesics above a closed geodesic $%
\gamma $ might be shorter than $l(\gamma )-c$ for some constant $c>0$
independent of $\gamma $. Comparing the angles of parallelism of $\frac{l}{2}
$ and $\frac{l-c}{2}$, one can see that $\Pi (\frac{l-c}{2})/\Pi (\frac{l}{2}%
)=arctan(e^{-l/2+c/2}) / arctan(e^{-l/2})$ is an increasing function of $l$
which is greater than 1 for each $c>0$. For $c\sim 1.44$, this function is
already greater than 2 for $l=3.72488..$ (the length of the shortest
non-simple geodesic on any hyperbolic surface, or twice the width of a
regular ideal quadrilateral). So a gap of around $1.44$ between the
projection length and the geodesic length of $\gamma$ would imply that the
self-intersection angles of $\gamma $ are larger than $2\Pi (l(\gamma )/2)$,
or twice the bound given in Corollary \ref{bounds for angles}.

\ 

We conclude by discussing the relations between the bounds obtained from our
results and from Gilman's result in \cite{G}. She considers a pair of
isometries $A$ and $B$ of the hyperbolic plane which generate a purely
hyperbolic subgroup of $SL(2,\mathbb{R})$. Denote their translation lengths
by $|A|$ and $|B|$, and suppose that their axes meet at an angle $\phi $.
Gilman shows that \ $sin(\phi )sinh(|A|/2)sinh(|B|/2)>1$. This lower bound
on the angle $\phi $ is stronger than the bound we obtain from Corollary \ref%
{angleoftwodistinctgeodesics}. Now consider the case when $A$ and $B$ are
conjugate, so that $|A|=|B|$. Gilman's result gives the inequality

\begin{equation*}
sin(\phi )>\frac{4}{e^{|A|}-2+e^{-|A|}}.
\end{equation*}

\ On the other hand, in the case of a single closed geodesic, Theorem \ref%
{orthogonalprojectionforsingleclosedgeodesic} implies that $\phi $ is larger
than the parallelism angle of $\frac{|A|}{2}$. Hence

\begin{equation*}
sin(\phi )>sin(\Pi(\frac{|A|}{2}))=sech(\frac{|A|}{2})=\frac{2}{e^{\frac{|A|%
}{2}}+e^{-\frac{|A|}{2}}}
\end{equation*}

\ 

which is much stronger than Gilman's bound if $|A|$ is large. The two bounds
are approximately equal when $|A|=2.2$, but as $|A|$ increases the ratio of
the two bounds tends to infinity.

\end{document}